\documentclass[12pt]{amsart}
\usepackage{amsmath}
\usepackage{amscd}
\usepackage{amssymb}
\usepackage{amsfonts}
\def\rd{\overset{\circ}{\mathrm{Ric}}}
\def\rdc{\overset{\circ}{\mathrm{R}}}
\def\Rm{{\mathrm {Rm}}}
\def\Ric{{\mathrm {Ric}}}
\def\R{{\mathrm {R}}}

\newcommand\tbbint{{-\mkern -16mu\int}}

\newcommand\dbbint{{-\mkern -19mu\int}}

\newcommand\bbint{
{\mathchoice{\dbbint}{\tbbint}{\tbbint}{\tbbint}}
}

\newtheorem{theorem}{Theorem}[section]
\newtheorem{lemma}[theorem]{Lemma}

\newtheorem{proposition}[theorem]{Proposition}

\theoremstyle{definition}

\theoremstyle{remark}
\newtheorem{remark}[theorem]{Remark}
\numberwithin{equation}{section}

\begin{document}
\title[Geometric inequalities and rigidity of shrinkers]
{Geometric inequalities and rigidity of gradient shrinking Ricci solitons}
\author{Jia-Yong Wu}
\address{Department of Mathematics and Newtouch Center for Mathematics, 
Shanghai University, Shanghai 200444, China}
\email{wujiayong@shu.edu.cn}
\thanks{}
\subjclass[2010]{Primary 53C25, 53C20; Secondary 53C24, 58J35.}
\dedicatory{}
\date{\today}

\keywords{shrinking Ricci soliton, Schr\"odinger operator, Sobolev inequality,
logarithmic Sobolev inequality, heat kernel, Faber-Krahn inequality, Nash inequality,
Rozenblum-Cwikel-Lieb inequality, eigenvalue, half Weyl tensor, rigidity}

\begin{abstract}
In this paper we prove that the Sobolev inequality, the logarithmic Sobolev inequality,
the Schr\"odinger heat kernel upper bound, the Faber-Krahn inequality, the Nash inequality and the
Rozenblum-Cwikel-Lieb inequality all equivalently exist on complete gradient shrinking Ricci
solitons. We also obtain some integral gap theorems for compact shrinking Ricci solitons.
\end{abstract}
\maketitle

\section{Introduction}\label{Int1}
In this paper we will investigate the equivalence of various geometric inequalities on
gradient shrinking Ricci solitons. As applications, We apply the Sobolev inequality to
give some integral gap theorems for compact shrinking Ricci solitons. Recall that an
$n$-dimensional Riemannian manifold $(M,g)$ is called a \emph{gradient shrinking Ricci
soliton or shrinkers} (see \cite{[Ham]}) if there exists a smooth function $f$ on $M$
such that the Ricci curvature $\text{Ric}$ and the Hessian of $f$ satisfy
\begin{align}\label{Eq0}
\Ric+\mathrm{Hess}\,f=\lambda g
\end{align}
for some positive number $\lambda$. Function $f$ is often called a \emph{potential} of the
shrinker. For simplicity, we often normalize $\lambda=\frac 12$ by scaling the metric $g$ so that
\begin{align}\label{Eq1}
\Ric +\mathrm{Hess}\,f=\frac 12g.
\end{align}
According to the work of \cite{[Ham],[CaNi]}, without loss of generality, adding $f$ by
a constant if necessary, we can assume that equation \eqref{Eq1} simultaneously
satisfies
\begin{equation}\label{Eq2}
\R+|\nabla f|^2=f \quad\mathrm{and}\quad (4\pi)^{-\frac n2}\int_Me^{-f} dv=e^{\mu},
\end{equation}
where $\R$ is the scalar curvature of $(M,g)$ and $\mu=\mu(g,1)$ is the entropy functional
of Perelman \cite{[Pe]}; see also the detailed explanation in \cite{[LLW]} or \cite{[W20], [W20b]}.
For a complete shrinker, $\mu$ is always finite, which has been confirmed in \cite{[LiWa]}. 
In particular, for the Euclidean space, we have $\mu=0$. In \cite{[LLW]}, Li, Li and 
Wang proved that $e^{\mu}$ is nearly equivalent to $V(p_0,1)$, i.e., the volume of 
geodesic ball $B(p_0,1)$ centered at point $p_0\in M$ and radius $1$. 
Here $p_0\in M $ is a point where $f$ attains its infimum, which always exists on
shrinkers but possibly is not unqiue; see \cite{[HaMu]}.

\vspace{.1in}

\emph{In this paper we always let the triple $(M, g, f)$ denote the shrinking gradient
Ricci soliton (or shrinker) with \eqref{Eq1} and \eqref{Eq2}.}

\vspace{.1in}

Shrinkers are natural extension of Einstein manifolds and can be regarded as the
critical point of Perelman's $\mathcal{W}$-functional \cite{[Pe]}. In addition, shrinkers
are self-similar solutions to the Ricci flow and naturally rise as singularity analysis of
the Ricci flow \cite{[Ham]}. For example, Enders, M\"uller and Topping \cite{[EMT]}
proved that the proper rescaling limits of a type-I singularity point always converge to
non-trivial shrinkers. At present, one of main issue in the Ricci flow theory is the
understanding on the geometry and classification of shrinkers. For dimensions 2 and 3, the
classification is complete by the works of \cite{[Ha88]}, \cite{[Ivey]}, \cite{[Pe]}, \cite{[NW]}
and \cite{[CCZ]}. For dimension 4 and higher, the classification remains open, though much
progress has been made. The interested reader can refer to \cite{[Cao]} for an excellent survey.

\vspace{.1in}

In recent years, many geometric and analytic results about shrinkers have been investigated.
Wylie \cite{[Wy]} proved that any complete shrinker has finite fundamental group (the compact
case due to Derdzi\'nski \cite{[De]}). Chen \cite{[Chen]} showed that the scalar curvature
$\mathrm{R}\geq0$; Pigola, Rimoldi and Setti \cite{[PiRS]} proved that $\mathrm{R}>0$ unless
$(M,g,f)$ is the Gaussian shrinker; Chow, Lu and Yang \cite{[CLY]} showed
that the scalar curvature of non-trivial shrinkers has at least quadratic decay of distance
function. Chen and Zhou \cite{[CaZh]} showed that the potential function $f$ is uniformly
equivalent to the distance squared; they \cite{[CaZh]} also showed that all shrinkers have
at most Euclidean volume growth by combining an observation of Munteanu \cite{[Mun]}.
Later, Munteanu and Wang \cite{[MuW14]} proved that shrinkers have at least linear volume
growth. These volume growth properties are similar to manifolds with nonnegative Ricci curvature.

\vspace{.1in}

On the other hand, Haslhofer and M\"uller \cite{[HaMu],[HaMu2]} proved a Cheeger-Gromov compactness
theorem of shrinker. Huang \cite{[Hua]} proved an $\epsilon$-regularity theorem for
$4$-dimensional shrinkers, which was later improved by Ge and Jiang \cite{[GeJi]}. Their
result gives an answer to Cheeger-Tian's question \cite{[CT]}. In \cite{[W15]}, the author
applied gradient estimate technique to prove a Liouvlle type theorem for ancient solutions
to the weighted heat equation on shrinkers. In \cite{[WW15], [WW16]}, P. Wu and the author applied
weighted heat kernel upper estimates to give a sharp weighted $L^1$-Liouville theorem
for weighted subharmonic functions on shrinkers.

\vspace{.1in}

By analyzing the Perelman's functional under the Ricci flow, Li and Wang \cite{[LiWa]}
obtained a sharp logarithmic Sobolev inequality on complete (possible non-compact)
shrinkers, which says that for any $\tau>0$,
\begin{equation}\label{LSI}
\int_M\varphi^2\ln \varphi^2dv\leq\tau\int_M\left(4|\nabla\varphi|^2+\mathrm{R}\varphi^2\right)dv
-\left(\mu+n+\frac n2\ln(4\pi\tau)\right).
\end{equation}
for any compactly supported locally Lipschitz function $\varphi$ with $\int_M\varphi^2dv=1$.
The equality case could be attained when $\tau=1$ (see Carrillo and Ni \cite{[CaNi]}). This
sharp inequality
is useful for understanding shrinkers. Indeed the author \cite{[W20b]} was able to apply
\eqref{LSI} to give sharp upper diameter bounds of compact shrinkers in terms of the integral
of scalar curvature. The author \cite{[W20]} also used \eqref{LSI} to study the Schrodinger heat
kernel on shrinkers. Here the definition of the Schr\"odinger heat kernel is similar to the
classical heat kernel. That is, for each $y\in M$, we say that $H^{\mathrm{R}}(x,y,t)$ is called
the Schrodinger heat kernel if $H^{\mathrm{R}}(x,y,t)=u(x,t)$ is a minimal positive smooth
solution of the Schr\"odinger heat equation
\[
-\Delta u+\tfrac{\R}{4} u+\partial_t u=0
\]
satisfying $\lim_{t\to0}u(x,t)=\delta_y(x)$, where $\delta_y(x)$ is
the delta function defined as
\[
\int_M\phi(x)\delta_y(x)dv=\phi(y)
\]
for any $\phi\in C_0^{\infty}(M)$. In general, the Schr\"odinger heat kernel always 
exists on compact shrinkers. For complete non-compact shrinkers, the Schr\"odinger 
heat kernel still exists, given that the scalar curvature $\mathrm{R}\ge 0$ on 
$(M,g,f)$. Indeed, the standard construction of Schr\"odinger heat kernels from 
compact domains can be extended to the limiting Schr\"odinger heat kernel. The 
Schr\"odinger heat kernel of shrinkers shares many kernel properties of the 
classical Laplacian heat kernel on manifolds; see \cite{[W20]}.

\vspace{.1in}

In \cite{[W20]}  the author applied \eqref{LSI} to prove that
\begin{equation}\label{upp2}
H^{\R}(x,y,t)\le\frac{e^{-\mu}}{(4\pi t)^{\frac n2}}
\end{equation}
for all $x,y\in M$ and $t>0$. We remark that this type upper bound of the conjugate heat
kernel under the Ricci flow was obtained by Li and Wang \cite{[LiWa]}. The author also
obtained its Gaussian type upper bounds by the iteration argument. That is, for any $\alpha>4$, the
author showed that there exists a constant $A=A(n,\alpha)$ depending on $n$ and $\alpha$ such that
\begin{equation}\label{upp}
H^{\R}(x,y,t)\le\frac{A e^{-\mu}}{(4\pi t)^{\frac n2}}\exp\left(-\frac{d^2(x,y)}{\alpha t}\right)
\end{equation}
for all $x,y\in M$ and $t>0$, where $d(x,y)$ denotes the geodesic distance between $x$
and $y$. Considering the classical Laplace heat kernel of Euclidean space, estimate
\eqref{upp} is obvious sharp. Moreover the heat kernel estimate is useful for analyzing
eigenvalues of the
Schr\"odinger operator. Indeed the author \cite{[W20]} used the upper bounds to get
lower bounds of their eigenvalues. Namely, for any open relatively compact set
$\Omega\subset M$, let $0<\lambda_1(\Omega)\le\lambda_2(\Omega)\le\ldots$ be the
Dirichlet eigenvalues of the Schr\"odinger operator in $\Omega$. Then we have
\begin{equation}\label{multie}
\lambda_k(\Omega)\ge\frac{2n\pi}{e}\left(\frac{k\,e^\mu}{V(\Omega)}\right)^{2/n},\quad k\ge1,
\end{equation}
where $V(\Omega)$ is the volume of $\Omega$. Recall that the
classical Weyl's asymptotic formula of the $k$-th Dirichlet eigenvalue of Laplacian
in open relatively compact set $\Omega\subset\mathbb{R}^n$ states that
\[
\lambda_k(\Omega)\sim c(n)\left(\frac{k}{V(\Omega)}\right)^{2/n},\quad k\to\infty,
\]
which indicates that \eqref{multie} is sharp for the exponent $2/n$. We remark that
by the Rozenblum-Cwikel-Lieb inequality \cite{[Ro],[Cw],[Lie]} (see also estimate
\eqref{RCLnum} in Section \ref{sec2pre}), we easily get that eigenvalues of the Schr\"odinger
operator $-\Delta+\frac{\mathrm{R}}{4}$ on non-trivial shrinkers are all positive.

\vspace{.1in}

Besides of the above results, combining \eqref{LSI} and the Markov semigroup technique
of Davies \cite{[Dav]}, Li and Wang \cite{[LiWa]} proved a local Sobolev inequality of
shrinkers. Namely, for each compactly supported locally Lipschitz function $\varphi$ in $M$,
\begin{equation}\label{sobo}
\left(\int_M \varphi^{\frac{2n}{n-2}}\,dv\right)^{\frac{n-2}{n}}\le C(n)e^{-\frac{2\mu}{n}}
\int_M\left(4|\nabla \varphi|^2+\R\varphi^2\right) dv
\end{equation}
for some constant $C(n)$ depending only on $n$. Here $e^{\mu}$ can be viewed as the volume
of unit geodesic ball and hence this Sobolev inequality is very similar to the classical Sobolev
inequality on manifolds, which plays an important role in some PDE ways. For example, P. Wu
and the author \cite{[WW19]} applied \eqref{sobo} to study the dimensional estimates for the
spaces of harmonic functions and Schr\"odinger functions with  polynomial growth. In \cite{[W19]},
the author used \eqref{sobo} to derive a mean value type inequality and further study the
analyticity in time for solutions of the heat equation on shrinkers.

\vspace{.1in}

In this paper we continue to study geometric inequalities and their relations on shrinkers.
We first show that the above geometric inequalities, the Nash inequality, and the
Rozenblum-Cwikel-Lieb inequality all equivalently exist on shrinkers, which may be
regarded as natural generalizations of the case of manifolds with nonnegative Ricci curvature
\cite{[Gr],[Sa],[Zhq]}.
\begin{theorem}\label{thmequ}
Let $(M,g, f)$ be an $n$-dimensional complete (compact or noncompact) shrinker.
The following six properties are equivalent up to constants.
\begin{itemize}
\item [(I)] The Sobolev inequality \eqref{sobo} holds.
\item [(II)] The logarithmic Sobolev inequality \eqref{LSI} holds.
\item [(III)] The Schr\"odinger heat kernel upper bound \eqref{upp} holds.
\item [(IV)] The Faber-Krahn inequality holds. That is, for all open relatively compact
set $\Omega\subset M$ with smooth boundary,
\[
\lambda_1(\Omega)\ge\frac{2n\pi}{e}\left(\frac{e^\mu}{V(\Omega)}\right)^{\frac 2n},
\]
where $\lambda_1(\Omega)$ is the lowest Dirichlet eigenvalue of the Schr\"odinger operator
in $\Omega$.

\item [(V)]
The Nash inequality holds. That is, there exists a constant $c(n)$ depending on $n$
such that
\begin{equation}\label{Nash}
\parallel\varphi\parallel^{2+\frac 4n}_2\le
c(n)e^{-\frac{2\mu}{n}}\parallel\varphi\parallel^{\frac 4n}_1
\int_M\left(4|\nabla \varphi|^2+\R\varphi^2\right)dv
\end{equation}
for any compactly supported locally Lipschitz function $\varphi$ in $M$.
\item [(VI)]
The Rozenblum-Cwikel-Lieb inequality holds. That is,
there exists a constant $c(n)$ depending on $n$
such that
\begin{equation}\label{RCL}
\mathcal{N}\left(-\Delta+\tfrac{\R}{4}+V\right)\le c(n)e^{-\mu}\int_M V^{\frac n2}_{-}dv,
\end{equation}
for any function $V\in L^1_{loc}(M)$, where $V_{-}:=\max\{0,-V\}\in L^{n/2}(M)$ is the
non-positive part of $V$, and $\mathcal{N}(A)$ is the number of non-positive $L^2$-eigenvalues
of the operator $A$, counting multiplicity.
\end{itemize}
\end{theorem}

\begin{remark}\label{inequvar}
We point out that (III) is equivalent to \eqref{upp2}. Indeed, (III) $\Rightarrow$ \eqref{upp2}
is obvious, while \eqref{upp2} $\Rightarrow$ (III) due to the work \cite{[W20]}. We also
point out that (IV) is equivalent to \eqref{multie}. Indeed, from Theorem \ref{thmequ},
we first have (IV) $\Rightarrow$ (III), and from the work \cite{[W20]}, we then know (III)
$\Rightarrow$ \eqref{multie}. Combining two parts finally yields (IV) $\Rightarrow$ \eqref{multie}.
The converse is trivial.
\end{remark}

\begin{remark}\label{inequvar2}
Estimates in (III) and (IV) are both sharp; see \cite{[W20]}. In addition \eqref{RCL}
is also sharp in some sense. Indeed on Gaussian shrinker
$(\mathbb{R}^n, \delta_{ij}, \frac{|x|^2}{4})$, where $\delta_{ij}$ is the standard
flat Euclidean metric, we know $\mathrm{R}=0$ and $\mu=0$.
For a large parameter $\alpha$, replacing $V$ by $\alpha V$ in \eqref{RCL},
\[
\alpha^{-\frac n2}\mathcal{N}\left(-\Delta+\alpha V\right)\le c(n)\int_{\mathbb{R}^n}V^{\frac n2}_{-}dv.
\]
In turns out that for any potential $V$ with $V_{-}\in L^{n/2}(\mathbb{R}^n)$ and $V\in L^1_{loc}(\mathbb{R}^n)$,
we have the Weyl asymptotics
\[
\lim_{\alpha\to \infty}\alpha^{-\frac n2}\mathcal{N}\left(-\Delta+\alpha V\right)=\frac{(2\sqrt{\pi})^{-n}}{\Gamma(1+\frac n2)}\int_{\mathbb{R}^n}V^{\frac n2}_{-}dv.
\]
This indicates that \eqref{RCL} is sharp in order in $\alpha$ for a class function of $V$.
\end{remark}

The proof strategy of Theorem \ref{thmequ} is as follows. Li and Wang \cite{[LiWa]} proved
that (II) holds on shrinkers and they confirmed that (II) $\Rightarrow$ (I). The
author \cite{[W20]} proved that (II) $\Rightarrow$ (III) $\Rightarrow$ (IV). The rest part
is the following.
\begin{itemize}
\item [(1)] We will apply the Jensen inequality to confirm (I) $\Rightarrow$ (II).
\item [(2)] We shall apply the Davies' argument \cite{[Dav]} and the Markov semigroup
to reprove Li-Wang's Sobolev inequality \eqref{sobo}, i.e., (III) $\Rightarrow$ (I).
In particular we will talk about the scope of Sobolev constant, which
will be useful to study gap theorems.
\item [(3)] We will apply the Schr\"odinger heat kernel and the level set method
to prove (IV) $\Rightarrow$ (III) by the approximation argument.
\item [(4)] We will use the H\"older inequality to prove (I) $\Rightarrow$ (V).
\item [(5)] By the approximation argument, we only need to apply the Schr\"odinger
heat kernel and analytical technique to prove (V) $\Rightarrow$ (III) with
Dirichlet condition.
\item [(6)] We will apply Schr\"odinger heat kernel properties and some functional theory
to prove (I) $\Leftrightarrow$ (VI).
\end{itemize}
We remark that the proof of (III) $\Rightarrow$ (I) will be separately provided
in Section \ref{sec2}; the proof of (I) $\Leftrightarrow$ (VI) will be given
in Section \ref{sec2pre}; the rest cases will be discussed in Section \ref{sec3}.

\vspace{.1in}

As applications, we will apply the Sobolev inequality of shrinkers to give integral
gap results for the Weyl tensor on compact shrinkers by adapting the proof
strategy of \cite{[Cati]}. To state the result, we fix some notations. We denote
by $W$, $\rd$ and  $V(M)$ the Weyl tensor, traceless Ricci tensor and the volume of
manifold $(M,g)$ respectively. The norm of a $(k,s)$-tensor $T$ of $(M,g)$ is defined
by $|T|^2_g:=g^{i_1m_1}\cdot\cdot\cdot g^{i_km_k}g_{j_1n_1}...g_{j_sj_s}
T^{j_1...j_s}_{i_1...i_k}T^{n_1...n_s}_{m_1...m_k}$.

\begin{theorem}\label{pingap}
Let $(M,g, f)$ be an $n$-dimensional, $4\le n\le 8$, compact shrinker. If
\begin{align*}
&\left(\int_M\Big|W+\frac{\sqrt{2}}{\sqrt{n}(n-2)}\rd\circ g \Big|^{\frac n2}dv\right)^{\frac 2n}+\left(\frac{\sqrt{n}}{2}-\sqrt{\frac{n-1}{2(n-2)}}\,\right)V(M)^{\frac 2n}\\
&\qquad<\left(\frac{1}{\sqrt{n}}-\frac 1n \sqrt{\frac{2(n-2)}{n-1}}\,\right)\frac{e^{\frac{2\mu}{n}}}{C(n)},
\end{align*}
where $\circ$ denotes the Kulkarni-Nomizu product and $C(n)$ is the constant in the Sobolev
inequality \eqref{sobo}, then $(M,g,f)$ is isometric to a quotient of the round sphere.
\end{theorem}
\begin{remark}
Chang, Gursky and Yang \cite{[CGY]} obtained integral gap results for compact manifolds
in terms of the Yamabe constant. Catino \cite{[Cati]} proved some integral gap results
for compact shrinkers, which was later improved in \cite{[FX]}. Our result involves the
the Sobolev constant of shrinkers rather than the Yamabe constant.
\end{remark}
The main ingredients of proving Theorem \ref{pingap} are Bochner-Weitzenb\"ock type formulas
for the norm of curvature tensors, curvature algebraic inequalities and Kato type
inequalities. The above pinching assumption is not true when $n\ge 9$. Indeed we will see
that constant $C(n)$ in Sobolev inequality \eqref{sobo} cannot be sufficiently small
(see Remark \ref{xian}), i.e.,
\[
C(n)\ge\frac{n-1}{2n(n-2)\pi e}.
\]
This range obviously affects the dimensional valid of the pinching assumption.
On the other hand, algebraic curvature inequalities and elliptic equation of
the norm of traceless Ricci tensor also restrict the choice of dimension
$n$, such as inequality \eqref{picond} in Section \ref{sec4}.

\vspace{.1in}

In particular, inspired by Cao-Tran's result \cite{[CaTr]}, we apply
the Sobolev inequality of shrinker to get an integral gap result for the half
Weyl tensor on compact four-dimensional shrinkers.
\begin{theorem}\label{hagap}
Let $(M,g, f)$ be a four-dimensional oriented compact shrinker.
Let $C(4)$ denote the constant in the Sobolev inequality \eqref{sobo} in $(M,g, f)$.
If
\[
\left(\int_M|W^{\pm}|^2dv\right)^{\frac{1}{2}}<\frac{e^{\frac{\mu}{2}}}{4\sqrt{6}C(4)}
\]
and
\[
\int_M|\delta W^{\pm}|^2dv\le\frac{1}{8}\int_M{\R}|W^{\pm}|^2dv,
\]
where $\delta$ is the divergence, then $W^{\pm}\equiv0$ and hence $(M^4,g, f)$
is isometric to a finite quotient of the round sphere or the complex projective
space.
\end{theorem}

\begin{remark}
For relevant notations of the half Weyl tensor $W^{\pm}$, see Section \ref{half}.
Gursky \cite{[Gu]} proved an integral pinching result for four-dimensional Einstein
manifolds involving the norm of half
Weyl tensor in terms of the Euler characteristic and the signature. Cao and Tran
\cite{[CaTr]} generalized Gursky's result to shrinkers. Our gap result depends
on the Sobolev constant of shrinkers rather than topological invariants.
\end{remark}

Inspired by Catino's result \cite{[CaTr]}, if we use the Yamabe constant instead of
the Sobolev inequality, then we have another integral gap result.
\begin{theorem}\label{hagap2}
Let $(M,g, f)$ be a four-dimensional oriented compact shrinker
satisfying \eqref{Eq0}. If
\begin{align*}
216\int_M|W^{\pm}|^2dv+12\int_M|\rd|^2dv\le\int_M{\R}^2dv
\end{align*}
and
\begin{align*}
\int_M|\delta W^{\pm}|^2dv\le\frac{1}{6}\int_M{\R}|W^{\pm}|^2dv,
\end{align*}
then $(M,g, f)$ is isometric to a finite quotient of the round sphere or the
complex projective space.
\end{theorem}

\begin{remark}
In Theorems
\ref{hagap} and \ref{hagap2}, the first assumption is an pinching condition of the half
Weyl tensor; while the second assumption is an pinching condition of the harmonic half
Weyl tensor. It is an interesting question if the second assumption is unnecessary.
\end{remark}

There are many gap results for Einstein manifolds, Ricci solitons and closed manifolds;
such as Catino and Mastrolia \cite{[CaMa]}, Hebey and Vaugon \cite{[HV]}, Li and Wang
\cite{[LiW9], [LiWa]}, Munteanu and Wang \cite{[MuW17]}, Petersen and Wylie \cite{[PW]},
Singer \cite{[Si]}, Tran \cite{[Tr]}, Zhang \cite{[Zhz]} and their references. In
this paper we provide a different gap criterion, which depends on the constant
$C(n)$ of Sobolev inequality. It is an interesting question to estimate a best
upper bound of $C(n)$ on shrinkers.

\vspace{.1in}

The structure of this paper is the following. In Section \ref{sec2}, we will
recall some basic results about algebraic inequalities of curvature tensors
and some formulas of shrinkers. In particular, we will reprove the Sobolev
inequality by the Schr\"odinger heat kernel upper bound. Meanwhile, we will
discuss the best Sobolev constant of shrinkers. In Section \ref{sec2pre},
we will apply the Schr\"odinger heat kernel to study the equivalence between
(I) and (VI) of Theorem \ref{thmequ}. In Section \ref{sec3}, we will prove
the rest cases of Theorem \ref{thmequ}. In Section \ref{sec4}, we will apply
the Sobolev inequality of shrinkers and Weitzenb\"ock formulas for curvature
tensors to prove Theorem \ref{pingap}. In Section \ref{half}, we will study
the gap results for half Weyl tensor. We shall prove Theorems \ref{hagap}
and \ref{hagap2}.

\vspace{.1in}

\textbf{Acknowledgement}.
This work is supported by the
NSFC (11671141) and the Natural Science Foundation of Shanghai (17ZR1412800).
The author thanks the anonymous referees for making valuable comments and 
pointing out some errors which helped to improve the exposition of the paper.


\section{Decomposition and Sobolev inequality}\label{sec2}
In this section we first give a brief introduction of curvature notations of the Riemannian
manifold $(M^n,g)$ and some algebraic inequalities of curvature tensors. Then we
review some geometric equations and formulas about shrinkers, especially for the Sobolev
inequality and its explicit coefficient. These results will be used in the following
sections. For more related results, see \cite{[LiWa]}, \cite{[W20]}.

We use $g_{ij}$ to be the local components of metric $g$ and its inverse by $g^{ij}$.
Let $\Rm$ be the $(4,0)$ Riemannian curvature tensor, whose local components denoted
by $\R_{ijkl}$. Let $\Ric$ denote the Ricci curvature with local components
$\R_{ik}=g^{jl}\R_{ijkl}$, and let $\R=g^{ik}\R_{ik}$ be the scalar curvature.
The traceless Ricci tensor is denoted by
\[
\rd=\Ric-\frac 1n\R g,
\]
whose local components
\[
\rdc_{ik}=\R_{ik}-\frac 1n\R g_{ik}.
\]
When $n\ge 4$, the Weyl tensor $W$ is defined by the orthogonal decomposition
\[
W=\Rm-\frac{\R}{2n(n-1)}g\circ g-\frac{1}{n-2} \rd\circ g,
\]
where $\circ$ denotes the Kulkarni-Nomizu product for two symmetric tensors $A$ and
$B$, which is defined as:
\[
(A\circ B)_{ijkl}=A_{ik}B_{jl}-A_{il}B_{jk}-A_{jk}B_{il}+A_{jl}B_{ik}.
\]
In local coordinates, we can write $W$ as
\begin{equation*}
\begin{aligned}
W_{ijkl}=&
\R_{ijkl}-\frac{1}{n-2}(g_{ik}\R_{jl}+g_{jl}\R_{ik}
-g_{il}\R_{jk}-g_{jk}\R_{il})\\
&+\frac{1}{(n-1)(n-2)}\R(g_{ik}g_{jl}-g_{il}g_{jk}).
\end{aligned}
\end{equation*}
The Weyl tensor has the same algebraic symmetries as the Riemannian curvature
tensor. It is well-known that the Weyl tensor is totally trace-free and it
is conformall invariant:
\[
W(e^{2\varphi}g)=e^{2\varphi}W(g)
\]
for any smooth function $\varphi$ on $M$.

\vspace{.1in}

In \cite{[Cati]}, Catino proved two algebraic curvature inequalities for any
$n$-dimensional Riemannian manifold, which will be used in the gap theorems.
\begin{lemma}\label{prlg}
Each $n$-dimensional Riemannian manifold $(M^n,g)$ satisfies the estimate
\[
\left|-W_{ijkl}\rdc_{ik}\rdc_{jl}+\frac{2}{n-2}\rdc_{ij}\rdc_{jk}\rdc_{ik}\right|\le \sqrt{\frac{n-2}{2(n-1)}}\left(|W|^2+\frac{8|\rd|^2}{n(n-2)}\right)^{\frac 12} |\rd|^2.
\]
\end{lemma}

\begin{lemma}\label{leal2}
On an $n$-dimensional Riemannian manifold $(M^n,g)$, there exists a positive constant $c(n)$ such that
\[
2W_{ijkl}W_{ipkq}W_{pjql}+\frac 12W_{ijkl}W_{klpq}W_{pqij} \le c(n)|W|^3.
\]
We can take $c(4)=\frac{\sqrt{6}}{4}$, $c(5)=1$, $c(6)=\frac{\sqrt{70}}{2\sqrt{3}}$ and $c(n)=\frac 52$ for $n\geq 7$.
\end{lemma}

\vspace{.1in}

On shrinker $(M,g,f)$, by Proposition 2.1 in \cite{[ELM]}, we have the following basic formulas,
which will be also used in the proof of gap theorems.
\begin{lemma}\label{formu}
Let $(M,g, f)$ be an $n$-dimensional complete shrinker. Then,
\[
\Delta f=\frac n2-\R,
\]
\[
\Delta_f\R=\R-2|\Ric|^2,
\]
\begin{align*}
\Delta_f\R_{ik}&=\R_{ik}-2W_{ijkl}\R_{jl}+\frac{2}{(n-1)(n-2)}\\
&\quad\times\Big(\R^2g_{ik}-n\R\R_{ik}+2(n-1)g^{mn}\R_{im}\R_{nk}-2(n-1)|\Ric|^2g_{ik}\Big),
\end{align*}
where $\Delta_f:=\Delta-\nabla f\cdot\nabla$.
\end{lemma}

In the end of this section, we will give the following Sobolev inequality of shrinkers,
which was proved by Li and Wang \cite{[LiWa]} by using the logarithmic Sobolev inequality.
Here we shall reprove it by using the upper bound of Schr\"odinger heat kernel. Meanwhile
we will provide an explicit Sobolev constant and discuss its range. This constant
will play a key role in proving gap results of shrinkers.
\begin{lemma}\label{lemm2}
Let $(M,g, f)$ be an $n$-dimensional complete shrinker. Then for each
compactly supported locally Lipschitz function $\varphi$ in $M$,
\begin{equation}\label{sobo2}
\left(\int_M\varphi^{\frac{2n}{n-2}}\,dv\right)^{\frac{n-2}{n}}
\le C(n)e^{-\frac{2\mu}{n}}\int_M\left(4|\nabla \varphi|^2+\R\varphi^2\right)dv
\end{equation}
for some dimensional constant $C(n)$. In particular, we can take
\[
C(n)=\frac{1}{\pi^2}\left(\frac{2}{n-2}\right)^{\frac{4}{n}}.
\]
\end{lemma}
\begin{remark}\label{xian}
For an $n$-sphere $S^n$ of radius $\sqrt{2(n-1)}$ with its standard metric, we have
$\mathrm{Ric}=\tfrac1 2 g$. Recall that on $S^n$, Aubin \cite{[Au]} (see also
Proposition 4.21 in \cite{[He]}) proved that for any $\varphi\in W^{1,2}(S^n)$,
\begin{equation}\label{sph}
\left(\int_{S^n}\varphi^{\frac{2n}{n-2}}\,dv\right)^{\frac{n-2}{n}}\le
\frac{8(n-1)}{n(n-2)}V(S^n)^{-2/n}\int_{S^n}|\nabla \varphi|^2dv+V(S^n)^{-2/n}\int_{S^n}\varphi^2 dv.
\end{equation}
This inequality is optimal in the sense that the two constants
$\frac{8(n-1)}{n(n-2)}V(S^n)^{-2/n}$ and $V(S^n)^{-2/n}$ can not be lowered.
On the other hand, from \eqref{Eq2}, we see that
\[
\mathrm{R}=f=\frac n2\quad \mathrm{and}\quad (4\pi e)^{-\frac n2}V(S^n)=e^{\mu}.
\]
Substituting them into \eqref{sobo2} and comparing with \eqref{sph}, we easily conclude that
\[
C(n)\ge\frac{n-1}{2n(n-2)\pi e}.
\]
\end{remark}

\begin{proof}[Proof of Lemma \ref{lemm2}]
We essentially follow the argument of Davies \cite{[Dav]} (see also \cite {[LiWa], [Zhq]}). Since
$H^{\mathrm{R}}=H^{\mathrm{R}}(x,y,t)$ is the Schr\"odinger heat kernel of the operator
$-\Delta+\frac{\R}{4}$, then
\[
\int_MH^{\mathrm{R}}(x,y,t)dv(y)\leq 1.
\]
In \cite{[W20]} we proved an upper of the Schr\"odinger heat kernel
\[
H^{\mathrm{R}}(x,y,t)\le\frac{e^{-\mu}}{(4\pi t)^{\frac n2}}
\]
for all $x,y\in M$ and $t>0$. In the following we will use this estimate to prove \eqref{sobo2}.

\vspace{.1in}

Now using the H\"older inequality, for any $u(x)\in L^2(M)$, we have
\begin{equation*}
\begin{aligned}
\parallel H^{\R}\ast u\parallel_{\infty}&=\sup_{x\in M}\left|\int_MH^{\R}(x,y,t) u(y)dv(y)\right|\\
&\le\sup_{x\in M}\left(\int_M(H^{\R})^2(x,y,t)dv(y)\right)^{\frac 12}\cdot\parallel u\parallel_2.
\end{aligned}
\end{equation*}
By \eqref{upp2}, we further have
\begin{equation*}
\begin{aligned}
\parallel H^{\R}\ast u\parallel_{\infty}
&\le\frac{e^{-\frac{\mu}{2}}}{(4\pi t)^{\frac n4}}\sup_{x\in M}\left(\int_M H^{\R}(x,y,t)dv(y)\right)^{\frac 12}\cdot\parallel u\parallel_2\\
&\le\frac{c^{1/2}_1}{t^{n/4}}\parallel u\parallel_2,
\end{aligned}
\end{equation*}
where $c_1=\frac{e^{-\mu}}{(4\pi)^{\frac n2}}$. Similarly, by the H\"older inequality,
for any $u(x)\in L^q(M)$, $q\in[1,n)$ and $\widetilde{q}=q/(q-1)$,
we may obtain another estimate
\begin{equation*}
\begin{aligned}
\parallel H^{\R}\ast u\parallel_{\infty}&\le\sup_{x\in M}\left(\int_M(H^{\R})^{\widetilde{q}}(x,y,t)dv(y)\right)^{1/\widetilde{q}}\cdot\parallel u\parallel_q\\
&\le\sup_{x\in M}\left(\sup_{y\in M}(H^{\R})^{\frac{1}{q-1}}(x,y,t)\int_MH^{\R}(x,y,t)dv(y)\right)^{1/\widetilde{q}}\cdot\parallel u\parallel_q\\
&\le\sup_{x\in M}\left(\sup_{y\in M}(H^{\mathrm{R}})^{\frac{1}{q-1}}(x,y,t)\right)^{1/\widetilde{q}}\cdot\sup_{x\in M}\left(\int_MH^{\R}(x,y,t)dv(y)\right)^{1/\widetilde{q}}\cdot\parallel u\parallel_q.
\end{aligned}
\end{equation*}
Using \eqref{upp2} again, for any $u(x)\in L^q(M)$ and $q\in[1,n)$,  we finally get
\begin{equation}\label{semineq}
\parallel H^{\R}\ast u\parallel_{\infty}\le\frac{c^{1/q}_1}{t^{\frac{n}{2q}}}\parallel u\parallel_q.
\end{equation}

Now we consider the integral operator
\[
L:=\left(\sqrt{-\Delta+\tfrac{\R}{4}}\right)^{-1}.
\]
Since $-\Delta+\tfrac{\R}{4}$ is a self-adjoint operator, by the eigenfunction expansion,
for any $u(x)\in C_0^\infty(M)$, we have
\begin{align*}
(Lu)(x)&=\Gamma(1/2)^{-1}\int^{\infty}_0 t^{-\frac 12}\big(e^{(\Delta-\mathrm{R}/4)t}u\big)(x,t)dt\\
&=\Gamma(1/2)^{-1}\int^{\infty}_0 t^{-\frac 12}\big(H^{\mathrm{R}}\ast u\big)(x,t)dt,
\end{align*}
where $e^{(\Delta-\mathrm{R}/4)t}u$ denotes the semigroup of $H^{\mathrm{R}}\ast u$.
Fix $T>0$, which will be determined later and let
\begin{equation*}
\begin{aligned}
(Lu)(x)&=\Gamma(\tfrac12)^{-1}\int^T_0 t^{-\frac 12}\big(H^{\mathrm{R}}\ast u\big)(x,t)dt
+\Gamma(\tfrac12)^{-1}\int^{\infty}_T t^{-\frac 12}\big(H^{\mathrm{R}}\ast u\big)(x,t)dt\\
&:=(L_1u)(x)+(L_2u)(x).
\end{aligned}
\end{equation*}

For any $\lambda>0$, we see that
\begin{equation}\label{set}
\left|\{x\big|\,|(Lu)(x)|\ge\lambda\}\right|\le\left|\{x\big|\,|(L_1u)(x)|\ge\lambda/2\}\right|
+\left|\{x\big|\,|(L_2u)(x)|>\lambda/2\}\right|.
\end{equation}
By \eqref{semineq} and the definition of $L_2u$, since $\Gamma(\tfrac12)=\sqrt{\pi}$, we have
\begin{equation*}
\begin{aligned}
\parallel L_2 u\parallel_{\infty}&\le\Gamma(1/2)^{-1}\int^{\infty}_T t^{-\frac 12}\left(\frac{c^{1/q}_1}{t^{\frac{n}{2q}}}\parallel u\parallel_q\right)dt\\
&=\frac{2qc_1^{1/q}}{(n-q)\sqrt{\pi}}\cdot T^{\frac 12-\frac{n}{2q}}\parallel u\parallel_q.
\end{aligned}
\end{equation*}
We now choose $T$ such that
\[
\frac{\lambda}{2}=\frac{2qc_1^{1/q}}{(n-q)\sqrt{\pi}}\cdot T^{\frac 12-\frac{n}{2q}}\parallel u\parallel_q.
\]
Then the set $\{x\big|\,|(L_2u)(x)|>\lambda/2\}=\emptyset$
and \eqref{set} becomes
\begin{equation*}
\begin{aligned}
\left|\{x\big|\,|(Lu)(x)|\ge\lambda\}\right|&\le\left|\{x\big|\,|(L_1u)(x)|\ge\lambda/2\}\right|\\
&\le(\lambda/2)^{-q}\int_M|(L_1u)(x)|^qdv(x).
\end{aligned}
\end{equation*}
We will estimate the right hand side of the above inequality. By the Minkowski inequality
for two measured spaces and the H\"older inequality, we get that
\begin{equation*}
\begin{aligned}
\parallel L_1u\parallel_q&\le\Gamma(\tfrac12)^{-1}\int^T_0 t^{-\frac 12}\parallel H^{\mathrm{R}}\ast u(\cdot,t)\parallel_q dt\\
&\le\Gamma(\tfrac12)^{-1}\int^T_0 t^{-\frac 12}\sup_{x\in M}\parallel H^{\mathrm{R}}(x,\cdot,t)\parallel_1\cdot\parallel u\parallel_q dt\\
&\le\frac{2}{\sqrt{\pi}}T^{1/2}\parallel u\parallel_q.
\end{aligned}
\end{equation*}
Hence,
\[
\left|\{x\big|\,|(Lu)(x)|\ge\lambda\}\right|\le\left(\frac{4}{\sqrt{\pi}}\right)^q\lambda^{-q}T^{q/2}\parallel u\parallel^q_q.
\]
According to the above choice of $T$, we have
\[
\left|\{x\big|\,|(Lu)(x)|\ge\lambda\}\right|\le c(n,q)c_1^{\frac{q}{n-q}}\lambda^{-r}\parallel u\parallel^r_q,
\]
where $r=qn/(n-q)$ and
\[
c(n,q)=\left(\frac{4}{\sqrt{\pi}}\right)^{\frac{nq}{n-q}}\left(\frac{q}{n-q}\right)^{\frac{q^2}{n-q}}.
\]

We see that for all $q\in[1,n)$, the linear operator $L$ could map the space $L^q(M)$ into the weak $L^r(M)$ space.
That is,
\begin{align*}
\parallel L u\parallel_{r,w}&\le c(n,q)^{\frac 1r}c_1^{\frac 1n}\parallel u\parallel_q\\
&=\frac{4}{\sqrt{\pi}}\left(\frac{q}{n-q}\right)^{\frac{q}{n}} c_1^{\frac 1n}\parallel u\parallel_q,
\end{align*}
where $\parallel \cdot\parallel_{r,w}$ denotes the weak $L^r$-norm.
For any $0<\epsilon<<1$, letting $q_1=q-\epsilon$, $q_2=q+\epsilon$,
$r_i=q_in/(n-q_i)$ ($i=1,2$), we indeed have
\[
\parallel L u\parallel_{r_i,w}\le\frac{4}{\sqrt{\pi}}\left(\frac{q_i}{n-q_i}\right)^{\frac{q_i}{n}} c_1^{\frac1n}\parallel u\parallel_{q_i}.
\]
Applying the Marcinkiewicz interpolation theorem to the above case, for any $0<t<1$, we get that
\[
\parallel L u\parallel_b\le\frac{4}{\sqrt{\pi}}\left[\left(\frac{q_1}{n-q_1}\right)^{\frac{q_1}{n}}\right]^t
\left[\left(\frac{q_2}{n-q_2}\right)^{\frac{q_2}{n}}\right]^{1-t}
c_1^{\frac1n}\parallel u\parallel_a,
\]
where
\[
\frac1a=\frac{t}{q_1}+\frac{1-t}{q_2},\quad \frac{1}{b}=\frac{t}{r_1}+\frac{1-t}{r_2}.
\]
Since the coefficient is continuous with respect to $\epsilon$ at $\epsilon=0$,
letting $\epsilon\to 0+$ and choosing $q=2$ and $p=2n/(n-2)$, then $a\to 2$ and
$b\to 2n/(n-2)$ and we finally get
\[
\parallel L u\parallel_p\le\frac{4}{\sqrt{\pi}}\left(\frac{2}{n-2}\right)^{\frac{2}{n}}c_1^{\frac1n}\parallel u\parallel_2,
\]
where $c_1=\frac{e^{-\mu}}{(4\pi)^{\frac n2}}$.
Let $\varphi=Lu$ and then $u=L^{-1}\varphi$ and
\begin{equation*}
\begin{aligned}
\parallel u\parallel^2_2&=\langle L^{-1}\varphi,L^{-1}\varphi\rangle\\
&=\langle L^{-2}\varphi,\varphi\rangle\\
&=\left\langle-\Delta\varphi+\tfrac{\R}{4} \varphi,\varphi\right\rangle\\
&=\int_M\left(|\nabla\varphi|^2+\tfrac{\R}{4}\varphi^2\right)dv.
\end{aligned}
\end{equation*}
Substituting this into the above inequality proves the theorem.
\end{proof}

In the above proof course , if we let $q_1=3$, $q_2=1$, $t=3/4$, $r_1=\frac{3n}{n-3}$
and $r_2=\frac{n}{n-1}$, we can also take
\[
C(n)=\frac{1}{\pi^2}\left(\frac{3}{n-3}\right)^{\frac{9}{2n}}\left(\frac{1}{n-1}\right)^{\frac{1}{2n}}.
\]
Obviously, when $4\le n\le 14$, this constant is bigger than the one in Lemma
\ref{lemm2} but when $n\ge15$, it is reverse. Naturally it is to ask

\vspace{.1in}

\noindent \textbf{Question}. \emph{For a complete gradient shrinking Ricci soliton $(M, g, f)$,
specially for the compact case, What is the best constant $C(n)$?
}


\section{Rozenblum-Cwikel-Lieb inequality}\label{sec2pre}
It is well-known that  the spectrum of the Laplacian $-\Delta$ on Riemannian manifold
$(M,g)$ is contained in the internal $[0,\infty]$. This can be proved by taking the Fourier
transform and using the Plancherel theorem. If one considers the Schr\"odinger operator
$-\Delta+V$ for some function $V$ on $M$, then $-\Delta+V$ may have some negative spectrum.
However, if we have some restriction on $V$, like decay conditions at infinity, we may
hope that the essential spectra of  $-\Delta+V$ and  $-\Delta$ coincide. In this case,
the negative spectrum of $-\Delta+V$ is a discrete set with possibly an accumulation
point at $0$ (if $0$ is indeed the bottom of the spectrum of $-\Delta$). It is an important
question in mathematical physics to estimate the number of these negative eigenvalues.
One of beautiful results about this question is the Rozenblum-Cwikel-Lieb (RCL)
inequality
\begin{equation}\label{RCLnum}
\mathcal{N}\left(-\Delta+V\right)\le c(n)\int_M V^{\frac n2}_{-}dv,
\end{equation}
where $V_{-}$ is the negative part of function $V\in L^1_{loc}(M)$, and $\mathcal{N}(A)$
is the number of non-positive $L^2$-eigenvalues of the operator $A$. The RCL inequality
was first established by Rozenblum \cite{[Ro]}, and it was independently found by Lieb
\cite{[Lie]} and Cwikel \cite{[Cw]} for $n\ge 3$. Afterwards, another remarkable proof
with sharper constant were provided by Li and Yau \cite{[LiY]}, where their proof relies
only on the positive of heat kernel and the Sobolev inequality.

\vspace{.1in}

On a shrinker $(M,g,f)$, one may would like to consider the special Schr\"odinger operator
\[
-\Delta^{\R}:=-\Delta+\tfrac{\R}{4}
\]
instead of the usual Laplacian. Since scalar curvature $\R\ge0$ on $(M,g,f)$,
by the RCL inequality, it is easy to see that its eigenvalues
are all nonnegative. If one considers a perturbation of $-\Delta^{\R}$ by a real-valued
potential $V$ and defines another Schr\"odinger operator $-\Delta^{\R}+V$, then the nonnegative
property of eigenvalues is not necessarily satisfied. Naturally one may ask:

\vspace{.1in}

\emph{What assumption on function $V$ will imply a bound on the number of negative
eigenvalues of $-\Delta^{\R}+V$ on shrinkers?}

\vspace{.1in}

In the following we will give an answer to this question. i.e., Theorem \ref{thmequ}:
(I) $\Rightarrow$ (VI). To prove this result, we start with an important proposition,
which is a key step of proving the RCL type inequality on shrinkers.
\begin{proposition}\label{proeige}
Let $D$ be a bounded domain in a shrinker $(M^n,g,f)$, where $n\ge 3$. Assume that $q(x)$ is a
positive function defined on $D$. Let $\lambda_k$ be the $k$-th eigenvalue of the equation
\[
-\Delta^{\R}\phi(x)=\lambda q(x)\phi(x)
\]
on $D$ with the Dirichlet boundary condition $\phi|_{\partial D}\equiv 0$. Then,
\[
\lambda_k^{n/2}\int_Dq^{n/2}(x)dv(x)\ge c(n)\,e^{\mu}\, k
\]
for some dimensional constant $c(n)$.
\end{proposition}
\begin{proof}[Proof of Proposition \ref{proeige}]
Inspired by the Li-Yau work \cite{[LiY]}, we consider the
``heat" kernel of the parabolic operator
\[
-\Delta^{\R}/{q}+\partial_t
\]
on shrinker $(M,g,f)$.
Let $\{\phi_i(x)\}^{\infty}_{i=1}$ be a set of orthonormal eigenfunctions such that
\[
-\Delta^{\R}\phi_i=\lambda_iq\phi_i,
\]
where $\lambda_i$ denote the eigenvalues of the corresponding eigenfunctions
$\{\phi_i(x)\}^{\infty}_{i=1}$ . Then the kernel $\widetilde{H}(x,y,t)$ of
$-\Delta^{\R}/q+\partial_t$ must have the following expression
\[
\widetilde{H}(x,y,t)=\sum^{\infty}_{i=1}e^{-\lambda_it}\phi_i(x)\phi_i(y).
\]
By the property of Schr\"odinger heat kernel $H^{\R}(x,y,t)$ (see \cite{[W20]}), we have
$\widetilde{H}(x,y,t)>0$ in the interior of $D\times D$ and $\widetilde{H}(x,y,t)\equiv0$
on $D\times \partial D$ and $\partial D\times D$ for any $t$.
At this time, the $L^2$-norm is given by the weighted volume $q(x)dv$ and
\[
\int_D\phi_i(x)\phi_j(x)q(x)dv=\delta_{ij}
\]
Since
\begin{equation}\label{defheat}
h(t):=\sum^{\infty}_{i=1}e^{-2\lambda_it}=\int_D\int_D\widetilde{H}^2(x,y,t)q(x)q(y)dv(x)dv(y),
\end{equation}
then we have
\begin{equation}
\begin{aligned}\label{heatev}
\frac{d h}{d t}&=2\int_D\int_D\widetilde{H}(x,y,t)\widetilde{H}_t(x,y,t)q(x)q(y)dv(x)dv(y)\\
&=2\int_D\int_D\widetilde{H}(x,y,t)\Delta^{\R}_{y}\widetilde{H}(x,y,t)q(x)dv(y)dv(x)\\
&=-2\int_D\int_D\left(|\nabla\widetilde{H}(x,y,t)|^2+\frac{\R}{4}\widetilde{H}^2(x,y,t)\right)q(x)dv(y)dv(x),
\end{aligned}
\end{equation}
where we used
\[
\left(\frac{\Delta^{\R}_y}{q(y)}-\partial_t\right)\widetilde{H}(x,y,t)=0.
\]
On the other hand, using the Cauchy-Schwarz inequality, we have
\begin{equation}
\begin{aligned}\label{evol}
h(t)&=\left[\int_Dq(x)\left(\int_D\widetilde{H}^{\frac{2n}{n-2}}(x,y,t)dv(y)\right)^{\frac{n-2}{n}}dv(x)\right]^{\frac{n}{n+2}}\\
&\quad\times\left[\int_Dq(x)\left(\int_D\widetilde{H}(x,y,t)q^{\frac{n+2}{4}}(y)dv(y)\right)^2dv(x)\right]^{\frac{2}{n+2}}.
\end{aligned}
\end{equation}
Let us now analyze the above inequality. Consider the quantity
\[
Q(x,t):=\int_D\widetilde{H}(x,y,t)q^{\frac{n+2}{4}}(y)dv(y)
\]
and it satisfies
\[
\left(\frac{\Delta^{\R}_x}{q(x)}-\partial_t\right)Q(x,t)=0
\]
with $Q(x,t)\equiv0$ on $\partial D$ for $t>0$ and $Q(x,0)=q^{\frac{n-2}{4}}(x)$.
We observe that
\begin{align*}
\partial_t\int_D q(x)Q^2(x,t)dv(x)&=2\int_Dq(x)Q(x,t)\partial_tQ(x,t)dv(x)\\
&=2\int_DQ(x,t)\Delta^{\R}_xQ(x,t)dv(x)\\
&=-2\left(\int_D|\nabla_xQ(x,t)|^2+\frac{\R}{4}Q^2(x,t)\right)dv(x)\\
&\le0,
\end{align*}
where we used the scalar curvature $\R\ge 0$ on shrinkers. This implies that
\begin{align*}
\int_D q(x)Q^2(x,t)dv(x)&\le\int_D q(x)Q^2(x,0)dv(x)\\
&=\int_Dq^{n/2}(x)dv(x).
\end{align*}
Using this, from \eqref{evol} we have
\begin{equation}\label{evol3}
h^{\frac{n+2}{n}}(t)\left(\int_Dq^{n/2}(x)dv\right)^{-\frac2n}\le\int_Dq(x)\left(\int_D\widetilde{H}^{\frac{2n}{n-2}}(x,y,t)dv(y)\right)^{\frac{n-2}{n}}dv(x).
\end{equation}
Recall that the Sobolev inequality \eqref{sobo} of shrinkers by letting $\varphi=\widetilde{H}(x,y,t)$ says that
\[
\left(\int_D |\widetilde{H}|^{\frac{2n}{n-2}}\,dv(y)\right)^{\frac{n-2}{n}}\le C(n)e^{-\frac{2\mu}{n}}
\int_D\left(4|\nabla\widetilde{H}|^2+\R\widetilde{H}^2\right) dv(y).
\]
Combining this with \eqref{evol3} and \eqref{heatev} yields
\[
\frac{d h}{d t}\le -\frac{e^{\frac{2\mu}{n}}}{2C(n)}\left(\int_Dq^{n/2}(x)dv(x)\right)^{-\frac2n}\cdot h^{\frac{n+2}{n}}(t).
\]
Dividing this by $h^{\frac{n+2}{n}}(t)$ and integrating with respect to $t$,
\[
h(t)\le (nC(n))^{\frac n2} e^{-\mu} \left(\int_Dq^{n/2}(x)dv(x)\right) t^{-\frac n2}.
\]
Combining this with \eqref{defheat}, we get
\[
\sum^{\infty}_{i=1}e^{-2\lambda_it}\le(nC(n))^{\frac n2} e^{-\mu} \left(\int_Dq^{n/2}(x)dv(x)\right) t^{-\frac n2}.
\]
Setting $t=\frac{n}{4\lambda_k}$, we conclude that
\[
\sum^{\infty}_{i=1}e^{-\frac{n\lambda_i}{2\lambda_k}}\le(nC(n))^{\frac n2} e^{-\mu}\int_Dq^{n/2}(x)dv(x) \left(\frac{n}{4\lambda_k}\right)^{-\frac n2}.
\]
Noticing that
\[
\sum^{\infty}_{i=1}e^{-\frac{n\lambda_i}{2\lambda_k}}\ge ke^{-n/2},
\]
so we have
\[
\lambda_k^{n/2}\int_Dq^{n/2}(x)dv(x)\ge c(n)\,e^{\mu}\, k
\]
for some dimensional constant $c(n)$.
\end{proof}

Now we will apply Proposition \ref{proeige} to prove Theorem \ref{thmequ}: (I) $\Rightarrow$ (VI).
\begin{proof}[Proof of Theorem \ref{thmequ}: (I) $\Rightarrow$ (VI)]
By the monotonicity of $\mathcal{N}\left(-\Delta^{\R}+V\right)$ with respect to the
function $V(x)$ on shrinker $(M,f,g)$, we may assume $V(x)\le 0$ by replacing $V(x)$
by $-V_{-}(x)$. Moreover $-V_{-}(x)$ can be approximated by a sequence of strictly
negative function. So we can assume $V(x)<0$ for all $x\in M$. By the exhausting
argument (see Lemma 3.2 in \cite{[OP]}), we only need to prove
\[
\mathcal{N}\left(-\Delta^{\R}+V\right)\le c(n)\int_D V_{-}^{\frac n2}dv
\]
for the equation
\begin{equation}\label{numest}
\left(-\Delta^{\R}+V\right)\phi=\lambda\phi
\end{equation}
with $\phi|_{\partial D}\equiv 0$ for any fixed domain $D\in M$.

It is easy to see that the number of non-positive eigenvalues
$\mathcal{N}\left(-\Delta^{\R}+V\right)$ for \eqref{numest}, counting multiplicity,
is equal to the number of eigenvalues less than $1$ for the case
in Proposition \ref{proeige} by considering
\[
q(x)=-V(x).
\]
Indeed, since $V(x)<0$, from the relation
\[
\frac{\int_D\left(|\nabla \phi|^2+\frac{\R}{4}\phi^2\right)dv+\int_DV\phi^2dv}{\int_D\phi^2dv}
=\frac{\int_D|V|\phi^2dv}{\int_D\phi^2dv}\left(\frac{\int_D\left(|\nabla \phi|^2+\frac{\R}{4}\phi^2\right)dv}{\int_D|V|\phi^2dv}-1\right),
\]
we conclude that the dimension of the subspace on which the left hand side is non-positive
is equal to the dimesion of the subspace on which the quadratic form
\[
\frac{\int_D\left(|\nabla \phi|^2+\frac{\R}{4}\phi^2\right)dv}{\int_D|V|\phi^2dv}
\]
associated to Proposition \ref{proeige} is not more than $1$.
Now we let $\lambda_k$ be the greatest eigenvalue which is not more than $1$. By
Proposition \ref{proeige}, we have
\begin{align*}
\int_D|V|^{n/2}dv(x)&\ge\lambda_k^{n/2}\int_D|V|^{n/2}dv(x)\\
&\ge c(n)\,e^{\mu}\, k\\
&\ge c(n)\,e^{\mu}\cdot\mathcal{N}(-\Delta^{\R}+V),
\end{align*}
which completes the proof of (I) $\Rightarrow$ (VI).
\end{proof}

Next we will prove the easy implication (VI) $\Rightarrow$ (I).
\begin{proof}[Proof of Theorem \ref{thmequ}: (VI) $\Rightarrow$ (I)]
We assume \eqref{RCL} holds for all potentials $V\in C^{\infty}_0(M)$.
Then for all $V\in C^{\infty}_0(M)$ satisfying
\[
\|V\|_{n/2}< c(n)e^{-\frac{2\mu}{n}},
\]
we know that $-\Delta^{\R}+V$ is a non-negative operator. That is,
\[
\int_M|\nabla \varphi|^2+\frac{\R}{4}\varphi^2 dv+\int_MV\varphi^2dv\ge 0
\]
for all such $V$ and all $\varphi\in C^{\infty}_0(M)$. In other words,
\[
\int_M|\nabla \varphi|^2+\frac{\R}{4}\varphi^2 dv\ge\sup_{V\in C^{\infty}_0(M)}\left\{\int_M-V\varphi^2dv\right\}
\]
satisfying
\[
\|V\|_{n/2}<c(n)e^{-\frac{2\mu}{n}}.
\]
Since the dual of $L^{n/2}(M)$ is $L^{n/(n-2)}(M)$,
the above functional inequality implies
\[
c(n)e^{-\frac{2\mu}{n}}\int_M\left(|\nabla \varphi|^2+\frac{\R}{4}\varphi^2\right)dv\ge\left(\int_M \varphi^{\frac{2n}{n-2}}\,dv\right)^{\frac{n-2}{n}}
\]
and Theorem \ref{thmequ} (I) follows.
\end{proof}

\section{Equivalence of geometric inequalities}\label{sec3}
In this section, we will give rest proofs of Theorem \ref{thmequ}. This part of theorem
mainly says that the (logarithmic) Sobolev inequality, the Schr\"odinger heat kernel upper
bound, the Faber-Krahn inequality and the Nash inequality are all equivalent up to possible
different constants. Notice that (II) $\Rightarrow$ (I) was proved in \cite{[LiWa]};
(II) $\Rightarrow$ (III) $\Rightarrow$ (IV) was proved in \cite{[W20]}. So we only need
to consider the following cases: (I) $\Rightarrow$ (II), (III) $\Rightarrow$ (I),
(IV) $\Rightarrow$ (III), (I) $\Rightarrow$ (V),  (V) $\Rightarrow$ (III).

\begin{proof}[Proof of Theorem \ref{thmequ}]
(I) $\Rightarrow$ (II): We may assume \eqref{sobo} holds on shrinkers. That is,
for each compactly supported locally Lipschitz function $\varphi$ in $(M,g,f)$,
\[
\left(\int_M\varphi^{\frac{2n}{n-2}}\,dv\right)^{\frac{n-2}{n}}
\le C(n)e^{-\frac{2\mu}{n}}\int_M\left(4|\nabla \varphi|^2+\R \varphi^2\right)dv
\]
for some dimensional constant $C(n)$. Given function $\varphi$ with $\|\varphi\|_2=1$,
we introduce the weighted measure $d\mu=\varphi^2dv$ on shrinker $(M,g,f)$, then
$\int_M d\mu=1$. Since function $\ln G$ is concave with respect to parameter $G$,
letting $G=\varphi^{q-2}$, where $q=\frac{2n}{n-2}$, and applying the Jensen inequality
\[
\int_M\ln Gd\mu\le\ln\left(\int_MGd\mu\right),
\]
we have that
\begin{align*}
\int_M(\ln \varphi^{q-2})\varphi^2dv&\le\ln\left(\int_M\varphi^{q-2}\varphi^2dv\right)\\
&=\ln \parallel\varphi\parallel^q_q,
\end{align*}
That is,
\begin{align*}
\int_M\varphi^2\ln\varphi dv&\le\frac{q}{q-2}\ln \parallel\varphi\parallel_q\\
&=\frac n2\ln \parallel\varphi\parallel_q.
\end{align*}
Combining this with the Sobolev inequality \eqref{sobo}, we get
\begin{align*}
\int_M\varphi^2\ln\varphi^2dv&\le\frac n2\ln \parallel\varphi\parallel^2_q\\
&\le\frac n2\ln\left[C(n)e^{-\frac{2\mu}{n}}\int_M\left(4|\nabla \varphi|^2+\mathrm{R}\varphi^2\right)dv\right]\\
&=\frac n2\ln C(n)-\mu+\frac n2\ln\left[\int_M\left(4|\nabla \varphi|^2+\mathrm{R}\varphi^2\right)dv\right].
\end{align*}
Using an elementary inequality:
\[
\ln x\leq \sigma x-(1+\ln\sigma)
\]
for any $\sigma>0$, the above estimate can be further reduced to
\[
\int_M\varphi^2\ln\varphi^2dv(x)\le\frac n2\ln C(n)-\mu+\frac{n\sigma}{2}\int_M\left(4|\nabla \varphi|^2+\mathrm{R}\,\varphi^2\right)dv-\frac n2(1+\ln\sigma).
\]
Setting $\tau=\frac{n\sigma}{2}$, we obtain
\[
\int_M\varphi^2\ln \varphi^2dv\le\tau\int_M\left(4|\nabla\varphi|^2+\mathrm{R}\varphi^2\right)dv-\mu-n-\frac n2\ln(4\pi\tau)
+\frac n2\ln(2ne\pi\cdot C(n))
\]
and hence (II) follows with possible different constants.

\vspace{.1in}

(III) $\Rightarrow$ (I): Since (III) implies \eqref{upp2} with different constants, then
\eqref{upp2} further implies (I) by following the proof of Lemma \ref{lemm2} in Section
\ref{sec2}.

\vspace{.1in}

(IV) $\Rightarrow$ (III): Since \eqref{upp2} is equivalent to (III), we only need to apply
(IV) to prove \eqref{upp2}. By the approximation argument, we only need to prove \eqref{upp2}
for the Dirichlet Schr\"odinger heat kernel $H_{\Omega}^\mathrm{R}(x,y,t)$ of any relatively
compact set $\Omega$ in $(M,g,f)$. In fact, let $\Omega_i$, $i=1,2,...$,  be a sequence of compact
exhaustion of $M$ such that $\overline{\Omega}_i\subset\Omega_{i+1}$ and $\cup_i\Omega_i=M$.
If we are able to prove \eqref{upp2} for the Dirichlet Schr\"odinger heat kernel
$H_{\Omega_i}^\mathrm{R}(x,y,t)$ for any $i$, then the result follows by letting
$i\to\infty$.

For a fixed point $y\in \Omega$, let $u=u(x,t)=H_{\Omega}^{\mathrm{R}}(x,y,t)$
and consider the integral
\[
I(t):=\int_{\Omega}u^2(x,t)dv.
\]
Then,
\begin{equation}\label{deriv}
I'(t)=2\int_{\Omega}uu_tdv=-2\int_{\Omega}\left(|\nabla u|^2+\frac14\R u^2\right)dv.
\end{equation}
For any positive number $s$, we have
\[
u^2\le (u-s)^2_{+}+2s u
\]
and therefore,
\[
I(t)\le\int_{\Omega}(u-s)_{+}^2dv+2s\int_{\Omega} u dv.
\]
Now for fixed $s,\, t>0$, consider the set \[
D(s,t):=\{x|x\in M, u(x,t)>s\}
\]
and its the first eigenvalue
\[
\lambda(D(s,t))=\inf_{0\neq \varphi\in C^{\infty}_0(D(s,t))}\frac{\int_{D(s,t)}(|\nabla\varphi|^2+\frac{\R}{4} \varphi^2)dv}{\int_{D(s,t)}\varphi^2dv}.
\]
Letting $\varphi=(u-s)_{+}$, then
\begin{align*}
\lambda(D(s,t))\left(I(t)-2s\right)&\le
\int_{D(s,t)}\left(|\nabla(u-s)_{+}|^2+\frac{\R}{4}\left((u-s)_{+}\right)^2\right)dv\\
&\le\int_{\Omega}\left(|\nabla u|^2+\frac{\R}{4} u^2\right)dv.
\end{align*}
Note that in the above first inequality, we threw away a positive term and used the Schr\"odinger
heat kernel property $\int_{\Omega}u(x,t)dv\le 1$. This property also indicates that
\[
V(D(s,t))\le s^{-1}.
\]
On the other hand, by the Faber-Krahn inequality, we have
\begin{align*}
\lambda(D(s,t))&\ge\frac{2n\pi}{e}\left(\frac{e^\mu}{V(D(s,t))}\right)^{\frac 2n}\\
&\ge\frac{2n\pi}{e}\left(e^\mu\cdot s\right)^{\frac 2n}.
\end{align*}
We remark that if the set $D(s,t)$ is not relatively compact, we can choose a sequence of relatively
compact sets which converges to it such that the Faber-Krahn inequality remains valid for $D(s,t)$.
Combining the above two inequalities, we obtain
\[
I(t)\le\frac{e^{1-2\mu/n}}{2n\pi}\int_{\Omega}\left(|\nabla u|^2+\frac{\R}{4} u^2\right)dv\cdot s^{-2/n}+2s.
\]
Minimizing the right hand side of the above inequality,
\[
I(t)\le c(n)e^{-\frac{2\mu}{n+2}}\left[\int_{\Omega}\left(|\nabla u|^2+\frac{\R}{4} u^2\right)dv\right]^{\frac{n}{n+2}}.
\]
Combining this with \eqref{deriv}, we have
\[
I'(t)\le c(n)e^{\frac{2\mu}{n}}I^{\frac{n+2}{n}}.
\]
Integrating this from $t/2$ to $t$,
\[
I(t)\le c(n)\frac{e^{-\mu}}{t^{n/2}}
\]
for $t>0$. In other words, we in fact get
\[
\int_{\Omega}H_{\Omega}^{\mathrm{R}}(x,y,t)H_{\Omega}^{\mathrm{R}}(x,y,t)dv(y)\le c(n)\frac{e^{-\mu}}{t^{n/2}}.
\]
By the Schr\"odinger heat kernel property, we indeed show that
\begin{align*}
H_{\Omega}^{\mathrm{R}}(x,x,2t)&=\int_{\Omega}H_{\Omega}^{\mathrm{R}}(x,y,t)H_{\Omega}^{\mathrm{R}}(y,x,t)dv(y)\\
&\le c(n)\frac{e^{-\mu}}{t^{n/2}}.
\end{align*}
This further implies
\begin{align*}
H_{\Omega}^{\mathrm{R}}(x,y,t)&=\int_{\Omega}H_{\Omega}^{\mathrm{R}}(x,z,t/2)H_{\Omega}^{\mathrm{R}}(z,y,t/2)dv(z)\\
&\le\left(\int_{\Omega}(H_{\Omega}^{\mathrm{R}})^2(x,z,t/2)dv(z)\right)^{1/2}
\left(\int_{\Omega}(H_{\Omega}^{\mathrm{R}})^2(z,y,t/2)dv(z)\right)^{1/2}\\
&=(H_{\Omega}^{\mathrm{R}})^{1/2}(x,x,t)(H_{\Omega}^{\mathrm{R}})^{1/2}(y,y,t)\\
&\le c(n)\frac{e^{-\mu}}{t^{n/2}}.
\end{align*}
Next we apply the same argument of proving Theorem 1.1 in \cite{[W20]}
to get an upper bound with a Gaussian exponential factor and finally
(III) follows.

\vspace{.1in}

(I) $\Rightarrow$ (V): We remark that the Nash inequality can be viewed as an interpolation
between the H\"older inequality and the Sobolev inequality. Assume that $(M,g,f)$ admits
\eqref{sobo}. By the H\"older inequality, for $p_1=\frac{n+2}{n-2}$ and $p_1=\frac{n+2}{4}$,
we have
\begin{align*}
\int_M\varphi^2dv&=\int_M\varphi^{\frac{2n}{n+2}}\varphi^{\frac{4}{n+2}}dv\\
&\le
\left(\int_M\varphi^{\frac{2n}{n+2}p_1}dv\right)^{1/{p_1}}\left(\int_M\varphi^{\frac{4}{n+2}p_2}dv\right)^{1/{p_2}}\\
&=\left(\int_M\varphi^{\frac{2n}{n-2}}dv\right)^{\frac{n-2}{n+2}}\left(\int_M|\varphi|dv\right)^{\frac{4}{n+2}}
\end{align*}
and hence,
\[
\parallel\varphi\parallel^{2+\frac 4n}_2\le
\left(\int_M\varphi^{\frac{2n}{n-2}}dv\right)^{\frac{n-2}{n}}\left(\int_M|\varphi|dv\right)^{\frac 4n}.
\]
Combining this with the Sobolev inequality \eqref{sobo} gives the Nash inequality \eqref{Nash}.

\vspace{.1in}

(V) $\Rightarrow$ (III):
Using the approximation argument, it suffices to prove \eqref{upp} for the Dirichlet
Schr\"odinger heat kernel $H_{\Omega}^\mathrm{R}(x,y,t)$ of any relatively compact set $\Omega$
in $(M,g,f)$.

For any $y\in M$, let $\varphi=\varphi(x,t)=H_{\Omega}^{\mathrm{R}}(x,y,t)$. Then
\begin{align*}
\frac{\partial}{\partial t}\left(\int_{\Omega}\varphi^2dv\right)&=\int_{\Omega}2\varphi\varphi_tdv\\
&=\int_{\Omega}2\varphi(\Delta\varphi-\frac 14\R\varphi)dv\\
&=-\frac 12\int_{\Omega}\left(4|\nabla \varphi|^2+\R\varphi^2\right)dv.
\end{align*}

Scaling function $\varphi$ such that $\parallel\varphi\parallel_1=1$, by our assumption,
we may assume the Nash inequality
\[
\parallel\varphi\parallel^{2+\frac 4n}_2\le c(n)e^{-\frac{2\mu}{n}}
\int_{\Omega}\left(4|\nabla \varphi|^2+\mathrm{R}\,\varphi^2\right)dv.
\]
Combining the above estimates, we have
\[
\frac{\partial}{\partial t}\left(\int_{\Omega}\varphi^2dv\right)\le-\frac{e^{\frac{2\mu}{n}}}{2c(n)}\parallel\varphi\parallel^{2+\frac 4n}_2.
\]
Let
\[
F(s):=\int_{\Omega}\varphi^2(x,s)dv,
\]
where $s\in(0,t]$. Then
\[
\frac{\partial}{\partial s}F(s)\le-\frac{e^{\frac{2\mu}{n}}}{2c(n)}F(s)^{1+\frac 2n}.
\]
Integrating it from $t/2$ to $t$ yields
\[
-\frac n2\left(F(t)^{-\frac 2n}-F(t/2)^{-\frac 2n}\right)\le-\frac{e^{\frac{2\mu}{n}}}{2c(n)}\cdot\frac t2,
\]
which implies that
\[
F(t)\le\left[2nc(n)\right]^{n/2}\frac{e^{-\mu}}{t^{n/2}}.
\]
This estimate is the same as $I(t)$ in the proof of the case (IV) $\Rightarrow$ (III).
Therefore we only use the same strategy of proving Theorem 1.1 in \cite{[W20]}
to get an upper bound with a Gaussian exponential factor and finally prove (III).
\end{proof}


\section{Gap result for Weyl tensor}\label{sec4}
In this section we will prove Theorem \ref{pingap} by using the arguments of \cite{[Cati],[FX]}.
We first recall the elliptic equation of the norm of traceless Ricci tensor on shrinkers,
which can be directly computed by Lemma \ref{formu} (see also Lemma 3.2 in \cite{[Cati]}).

\begin{lemma}\label{formulas2}
If $(M^n,g, f)$ be an $n$-dimensional shrinker satisfying \eqref{Eq1}, then
\[
\frac 12\Delta_f |\rd|^2=|\nabla\rd|^2 +|\rd|^2 -2W_{ijkl}\rdc_{ik}\rdc_{jl}+\frac{4}{n-2}\rdc_{ij}\rdc_{jk}\rdc_{ik}-\frac{2(n-2)}{n(n-1)}\mathrm{R}|\rd|^2.
\]
\end{lemma}

In the following we will apply the similar arguments of \cite{[Cati],[FX]}
to prove gap theorems on shrinkers. In our case, we need to carefully
deal with the constant $C(n)$ of Sobolev inequality \eqref{sobo}.
\begin{proof}[Proof of Theorem \ref{pingap}]
By Lemma \ref{formulas2}, using
\[
\frac 12\Delta|\rd|^2=|\nabla|\rd||^2+|\rd|\Delta|\rd|
\]
and the Kato inequality
\[
|\nabla\rd|^2\geq |\nabla|\rd||^2
\]
at each point where $|\rd|\neq 0$, we obtain
\begin{equation}\label{ricci}
|\rd|\Delta|\rd|\ge |\rd|^2-2W_{ijkl}\rdc_{ik}\rdc_{jl}+\frac {4}{n-2}\rdc_{ij}\rdc_{jk}\rdc_{ki}-\frac {2(n-2)}{n(n-1)}\R |\rd|^2+\frac 12\langle\nabla f, \nabla|\rd|^2\rangle.
\end{equation}
To simplify the notation, we let $u:=|\rd|$. Then for any positive number $s$,
which will be determined later, by \eqref{ricci}, we compute that
\begin{align*}
u^s\Delta u^s&=s(s-1)u^{2s-2}|\nabla u|^2+s u^{2s-1}\Delta u\\
&=\left(1-\frac{1}{s}\right)|\nabla u^s|^2+s u^{2s-2}\,u\Delta u\\
&\ge\left(1-\frac{1}{s}\right)|\nabla u^s|^2+su^{2s}+s\left(-2W_{ijkl}\rdc_{ik}\rdc_{jl}+\frac{4}{n-2}\rdc_{ij}\rdc_{jk}\rdc_{ki}\right)u^{2s-2}\\
&\quad-\frac {2(n-2)}{n(n-1)}s \R u^{2s}+\frac{s}{2} u^{2s-2}\langle\nabla f, \nabla u^2\rangle.
\end{align*}
Using Lemma \ref{prlg}, we further have
\begin{align*}
u^s\Delta u^s
&\ge\left(1-\frac{1}{s}\right)|\nabla u^s|^2+s u^{2s}-\sqrt{\frac {2(n-2)}{n-1}}s\left(|W|^2+\frac{8u^2}{n(n-2)}\right)^{\frac 12}u^{2s}\\
&\quad-\frac {2(n-2)}{n(n-1)}s\R u^{2s}+\frac 12\langle\nabla f, \nabla u^{2s}\rangle.
\end{align*}

Since $M^n$ is closed, integrating by parts over $M^n$ and using the equality
\[
\Delta f=\frac n2-\R
\]
from Lemma \ref{formu},  we have that
\begin{align*}
0&\ge\left(2-\frac{1}{s}\right)\int_M |\nabla u^s|^2dv+s\int_M u^{2s}dv
-\sqrt{\frac{2(n-2)}{n-1}}s\int_M\left(|W|^2+\frac{8 u^2}{n(n-2)}\right)^{\frac 12}u^{2s}dv\\
&\quad-\frac{2(n-2)}{n(n-1)}s\int_M \R u^{2s}dv -\frac 12 \int_M u^{2s}\Delta fdv\\
&=\left(2-\frac{1}{s}\right)\int_M |\nabla u^s|^2dv
-\sqrt{\frac {2(n-2)}{n-1}}s\int_M\left(|W|^2+\frac{8 u^2}{n(n-2)}\right)^{\frac 12}u^{2s}dv\\
&\quad-\left(\frac n4-s\right)\int_M u^{2s}dv+\frac{n(n-1)-4(n-2)s}{2n(n-1)}\int_M \R u^{2s}dv.
\end{align*}
For $2-{1}/{s}>0$, by the Sobolev inequality of shrinker using $\varphi=u^s$
\begin{equation}\label{sobinequ}
\int_M|\nabla u^s|^2dv\ge\frac{e^{\frac{2\mu}{n}}}{4C(n)}\left(\int_Mu^{\frac{2ns}{n-2}}\,dv\right)^{\frac{n-2}{n}}
-\frac 14\int_M\R u^{2s}dv,
\end{equation}
the above inequality becomes
\begin{align*}
0&\ge\left(2-\frac{1}{s}\right)\frac{e^{\frac{2\mu}{n}}}{4C(n)}\left(\int_M  u^{\frac{2ns}{n-2}}dv\right)^{\frac{n-2}{n}}
-\sqrt{\frac {2(n-2)}{n-1}}s\int_M\left(|W|^2+\frac{8 u^2}{n(n-2)}\right)^{\frac 12}u^{2s}dv\\
&\quad-\left(\frac n4-s\right)\int_M u^{2s}dv+\frac{n(n-1)-8(n-2)s^2}{4n(n-1)s}\int_M \R u^{2s}dv.
\end{align*}
By the H\"{o}lder inequality, for $n-4s\ge0$, we get that
\begin{align*}
0&\ge\left\{\left(2-\frac{1}{s}\right)\frac{e^{\frac{2\mu}{n}}}{4C(n)}-\sqrt{\frac {2(n-2)}{n-1}}s\left[\int_M\left(|W|^2+\frac{8 u^2}{n(n-2)}\right)^{\frac n4}dv\right]^{\frac 2n}-\left(\frac n4-s\right)V(M)^{\frac 2n}\right\}\\
&\quad\times\left(\int_M  u^{\frac{2ns}{n-2}}dv\right)^{\frac{n-2}{n}}
+\frac{n(n-1)-8(n-2)s^2}{4n(n-1)s}\int_M \R u^{2s}dv.
\end{align*}
Now we choose
\[
s=\sqrt{\frac{n(n-1)}{8(n-2)}}\in\left(\frac 12,\,\,\frac n4\right],
\]
and the last term of the above inequality vanishes. Moreover, notice that the
curvature integral assumption of theorem is equivalent to
\begin{equation}\label{pinchco}
\left(2-\frac{1}{s}\right)\frac{e^{\frac{2\mu}{n}}}{4C(n)}-\sqrt{\frac {2(n-2)}{n-1}}s\left[\int_M \left(|W|^2+\frac{8 u^2}{n(n-2)}\right)^{\frac n4}dv\right]^{\frac 2n}-\left(\frac n4-s\right)V(M)^{\frac 2n}>0,
\end{equation}
where we used the equality
\[
\left|W+\frac{\sqrt{2}}{\sqrt{n}(n-2)}\rd\circ g\right|^2=|W|^2+\frac{8}{n(n-2)}|\rd|^2
\]
due to the totally trace-free tensor $W$. Therefore, we conclude that $|\rd|\equiv 0$
and $(M,g, f)$ is Einstein.

Now we have $\Ric=\frac12 g$. By \eqref{Eq2}, we know
\[
f=\frac n2 \quad\mathrm{and}\quad (4\pi e)^{-\frac n2}V(M)=e^{\mu}
\]
and the pinching condition \eqref{pinchco} reduces to
\begin{equation}\label{pic}
\left({\bbint}_M|W|^{\frac n2}dv\right)^{\frac 2n}\le\epsilon_1(n):=
\sqrt{\frac{n-1}{32(n-2)}}\left[\left(\frac{2}{s}-\frac{1}{s^2}\right)\frac{C(n)^{-1}}{4\pi e}+4-\frac{n}{s}\right],
\end{equation}
where ${\bbint}_M$ denotes the average of the integration, i.e.,
\[
{\bbint}_M|W|^{\frac n2}dv=\frac{1}{V(M)}\int_M|W|^{\frac n2}dv.
\]
By Remark \ref{xian}, we see
\[
C(n)\ge\frac{n-1}{2n(n-2)\pi e}
\]
and hence
\[
\epsilon_1(n)\leq\sqrt{\frac{n-1}{32(n-2)}}\left[\left(\frac{2}{s}-\frac{1}{s^2}\right)\frac{n(n-2)}{2(n-1)}
+4-\frac{n}{s}\right].
\]
Notice that the right hand side of the above inequality is nonnegative if
\begin{equation}\label{picond}
s\ge\frac{n+\sqrt{n^2+8n(n-1)(n-2)}}{8(n-1)}.
\end{equation}
Since $s=\sqrt{\frac{n(n-1)}{8(n-2)}}$, it is easy to check that the above inequality
only holds only when $4\le n\le 8$. We then carefully calculate the constants as follows:
\begin{align*}
\epsilon_1(4)&\le\frac{5\sqrt{3}-6}{18}\approx .1478,
\quad\epsilon_1(5)\le\frac{7\sqrt{6}-6\sqrt{5}}{48}\approx .0778,
\quad\epsilon_1(6)\le \frac{9\sqrt{10}-10\sqrt{6}}{100}\approx .0397,\\
\epsilon_1(7)&\le\frac{\sqrt{3}(11-\sqrt{105})}{36\sqrt{5}}\approx .0162,
\quad\epsilon_1(8)\le\frac{13\sqrt{21}-42\sqrt{2}}{294}\approx 6\times 10^{-4}.
\end{align*}
Obviously, these constants in \eqref{pic} are strictly smaller than those in the
following Proposition \ref{Eisrig}, and hence $(M,g,f)$ is isometric to a quotient of the sphere.
\end{proof}

The following result is an gap result for Einstein manifolds, which was essentially proved
by Catino \cite{[Cati]}. The present version of pinching constants is a little better than
Theorem 3.3 in \cite{[Cati]} and seems to be more suitable to our applicable purpose.

\begin{proposition}\label{Eisrig}
Let $(M^{n},g)$ be an $n$-dimensional Einstein manifold with $\mathrm{Ric}=k g$, where
$k>0$ is a constant. There exists a positive constant $\epsilon_2(n)$ depending only on $n$
such that if
\[
\left(\bbint_M|W|^{\frac n2}dv\right)^{\frac 2n}<\epsilon_2(n),
\]
where ${\bbint}_M|W|^{\frac n2}dv=\frac{1}{V(M)}\int_M|W|^{\frac n2}dv$,
then $(M^{n},g)$ is isometric to a quotient of the round sphere with radius $\sqrt{\frac{n-1}{k}}$.
We can take $\epsilon_2(4)=\frac{14}{5\sqrt{6}}k$, $\epsilon_2(5)=\frac{4}{5}k$,
$\epsilon_2(6)=\frac{16\sqrt{3}}{9\sqrt{70}}k$, $\epsilon_2(7)=\frac{49}{125}k$,
$\epsilon_2(8)=\frac{267}{625}k$, $\epsilon_2(9)=\frac{23}{50}k$ and
$\epsilon_2(n)=\frac{2n}{5(n-1)}k$  if $n\ge 10$.
\end{proposition}

\begin{proof}[Proof of Proposition \ref{Eisrig}]
Following the argument in \cite{[HV],[Cati]}, we have the Bochner type formula for $|W|^2$,
\[
\frac{1}{2}\Delta|W|^2=|\nabla W|^2+2k|W|^2-2\left(2W_{ijkl}W_{ipkq}W_{pjql}+\frac{1}{2}W_{ijkl}W_{klpq}W_{pqij}\right).
\]
Since
\[
\frac 12\Delta|W|^2=|\nabla|W||^2+|W|\Delta|W|,
\]
then we have
\[
|W|\Delta|W|=|\nabla W|^2-|\nabla|W||^2+2k|W|^2-2\left(2W_{ijkl}W_{ipkq}W_{pjql}+\frac{1}{2}W_{ijkl}W_{klpq}W_{pqij}\right).
\]
Using Lemma \ref{leal2}
and the refined Kato inequality
\[
|\nabla W|^2\ge\frac{n+1}{n-1}|\nabla|W||^2
\]
at every point where $|W|\neq 0$, we get
\begin{equation}\label{wye}
|W|\Delta|W|\ge\frac{2}{n-1}|\nabla W|^2+2k|W|^2-2c(n)|W|^3,
\end{equation}
where $c(n)$ is a dimensional constant, which is defined by $c(4)=\frac{\sqrt{6}}{4}$, $c(5)=1$, $c(6)=\frac{\sqrt{70}}{2\sqrt{3}}$ and $c(n)=\frac{5}{2}$ for $n\geq 7$.

Similar to the preceding computation, we consider the quantity $u^s:=|W|^s$,
where $s$ is a positive number, which will be chosen later. Using \eqref{wye}, we compute
\begin{align*}
u^s\Delta u^s&=\left(1-\frac{1}{s}\right)|\nabla u^s|^2+s u^{2s-2}\,u\Delta u\\
&\ge\left(1-\frac{1}{s}\right)|\nabla u^s|^2+\frac{2s}{n-1}u^{2s-2}|\nabla u|^2+2ksu^{2s}-2c(n)su^{2s+1}\\
&=\left(1-\frac{n-3}{(n-1)s}\right)|\nabla u^s|^2+2ksu^{2s}-2c(n)su^{2s+1}.
\end{align*}
Since $M^n$ is closed, integrating the above inequality over $M^n$ yields
\begin{equation}\label{wye2}
0\ge\left(2-\frac{n-3}{(n-1)s}\right)\int_M|\nabla u^s|^2dv+2ks\int_Mu^{2s}dv-2c(n)s\int_{M}u^{2s+1}dv.
\end{equation}
Recall that Ilias \cite{[Il]} proved the Sobolev inequality of manifolds satisfying
$\mathrm{Ric}\ge kg$, where $k>0$, (by letting $f=|W|^s=u^s$ in \cite{[Il]}), which can be stated
that in our setting
\[
\left(\int_M u^{\frac{2ns}{n-2}}\,dv\right)^{\frac{n-2}{n}}\le
\frac{4(n-1)}{n(n-2)k}V(M)^{-2/n}\int_M|\nabla u^s|^2dv+V(M)^{-2/n}\int_M u^{2s} dv.
\]
Applying the H\"older inequality and the above Sobolev inequality, \eqref{wye2} becomes
\begin{align*}
0\ge&\left(2-\frac{n-3}{(n-1)s}\right)\int_M|\nabla u^s|^2dv+2ks\int_Mu^{2s}dv
-2c(n)s\left(\int_{M}u^{\frac n2}dv\right)^{\frac 2n}\left(\int_Mu^{\frac{2ns}{n-2}}dv\right)^{\frac{n-2}{n}}\\
\ge&\left(2-\frac{n-3}{(n-1)s}\right)\int_M|\nabla u^s|^2dv+2ks\int_Mu^{2s}dv\\
&-2c(n)sV(M)^{-2/n}\left(\int_Mu^{\frac n2}dv\right)^{\frac 2n}\left(\frac{4(n-1)}{n(n-2)k}\int_M|\nabla u^s|^2dv
+\int_Mu^{2s}dv \right).
\end{align*}
By the proposition assumption, we have the following equivalent form
\[
\left(\int_Mu^{\frac n2}dv\right)^{\frac 2n}< \epsilon_2(n)V(M)^{2/n}.
\]
Therefore, for $s>0$, if $\epsilon_2(n)$ satisfies
\begin{equation*}
\begin{cases}
2-\frac{n-3}{(n-1)s}-8c(n)s\epsilon_2(n)\frac{(n-1)}{n(n-2)k} \,\geq\, 0,\\
2ks-2c(n)\epsilon_2(n) \,\geq\, 0,
\end{cases}
\end{equation*}
we immediately have $W\equiv 0$ and $g$ has constant positive sectional curvature.
Here we give explicit constants such that the above two inequalities holds.

When $n=4$ and $s=\frac{7}{10}$, since $c(4)=\frac{\sqrt{6}}{4}$, we can take $\epsilon_2(4)=\frac{14}{5\sqrt{6}}k$.

When $n=5$ and $s=\frac{4}{5}$, since $c(5)=1$, we can take $\epsilon_2(5)=\frac{4}{5}k$.

When $n=6$ and $s=\frac{9}{10}$, since $c(6)=\frac{\sqrt{70}}{2\sqrt{3}}$, we can take $\epsilon_2(6)=\frac{16\sqrt{3}}{9\sqrt{70}}k$.

When $n=7$ and $s=\frac{49}{50}$, since $c(7)=\frac{5}{2}$, we can take $\epsilon_2(7)=\frac{49}{125}k$.

When $n=8$ and $s=\frac{267}{250}$, since $c(8)=\frac{5}{2}$, we can take $\epsilon_2(8)=\frac{267}{625}k$.

When $n=9$ and $s=\frac{23}{20}$, since $c(9)=\frac{5}{2}$, we can take $\epsilon_2(9)=\frac{23}{50}k$.

When $n\ge10$ and $s=\frac{n}{n-1}$, since $c(n)=\frac{5}{2}$, we can take $\epsilon_2(n)=\frac{2n}{5(n-1)}k$.
\end{proof}
\section{Gap result for half Weyl tensor}\label{half}
In this section we will apply the similar argument of Section \ref{sec4} to talk about an gap
phenomenon for shrinkers under the integral condition of half Weyl tensor.

Recall that, on an oriented $4$-dimensional manifold $(M,g)$, the bundle of $2$-forms $\wedge^2 TM$
can be decomposed as a direct sum
\[
\wedge^2 TM=\wedge^+ M\oplus\wedge^-M,
\]
where $\wedge^{\pm} M$ is the $(\pm 1)$-eigenspace of the Hodge star operator
\[
\star: \wedge^2 TM \rightarrow \wedge^2 TM.
\]
Sections of $\wedge^{\pm} M$ are called self-dual and anti-self-dual 2-forms.
Moreover, the Hodge star operator can further induce a decomposition for the curvature operator
$\mathfrak{R}: \wedge^2 TM \rightarrow \wedge^2 TM$,
\begin{equation*}
\mathfrak{R} =\left( \begin{array}{cc}
\frac{\R}{12}g+W^+ & \rd \\
\rd & \frac{\R}{12}g+W^-
\end{array} \right),
\end{equation*}
for the half Weyl tensor $W^{\pm}$ the restriction of the Weyl tensor $W$ to $\wedge^{\pm} M$.
Here $W^{\pm}:\wedge^{\pm}M \rightarrow \wedge^{\pm}M $ are also called the self-dual part
and anti-self-dual part of the Weyl tensor $W$, respectively.

In local coordinates, let $\{e_i\}^4_{i=1}$ be an oriented
orthonormal basis of tangent bundle $T M$. For any pair $(ij)$, $1\leq i\neq j\leq 4$,
let $(i'j')$ denote the dual of $(ij)$, i.e., the pair such that
\[
e_i\wedge e_j\pm e_{i'}\wedge e_{j'}\in \wedge^{\pm}M.
\]
In other words,
$(iji'j')=\sigma(1234)$ for some even permutation $\sigma\in S_4$. For the
Weyl tensor $W$, its (anti-)self-dual part is
\begin{equation*}
W^{\pm}_{ijkl}=\frac{1}{4}\left(W_{ijkl}\pm W_{ijk'l'}\pm W_{i'j'kl}+W_{i'j'k'l'}\right).
\end{equation*}
It is easy to check that
\[
W^{\pm}_{ijkl}=\pm W^{\pm}_{ijk'l'}=\pm W^{\pm}_{i'j'kl}=W^{\pm}_{i'j'k'l'}=\frac{1}{2}(W_{ijkl}{\pm}W_{ijk'l'}).
\]

On shrinkers, we have the following Weitzenb\"ock formula for the half Weyl tensor $W^{\pm}$
(see \cite{[CaTr]} or its generalization \cite{[Wp]}), and it is useful for analyzing the structure
of shrinkers, see for example \cite{[WWW]}.
\begin{lemma} \label{weitenbock}
Let $(M, g, f)$ be a four-dimensional shrinker satisfying \eqref{Eq0}. Then
\[
\frac 12\Delta_f|W^{\pm}|^2=|\nabla W^{\pm}|^2+2\lambda|W^{\pm}|^2-18\det
W^{\pm}-\frac 12\langle(\overset{\circ}{\mathrm{Ric}}\circ\overset{\circ}{\mathrm{Ric}})^{\pm},W^{\pm}\rangle.
\]
\end{lemma}

In this paper we can apply Lemma \ref{weitenbock} to prove Theorem \ref{hagap} in the introduction.

\begin{proof}[Proof of Theorem \ref{hagap}]
By Lemma \ref{weitenbock}, using the following algebraic inequality observed by Cao and Tran \cite{[CaTr]}
\[
\det W^{\pm}\le \frac{\sqrt{6}}{18}|W^{\pm}|^3
\]
and the Kato inequality
\[
|\nabla W^{\pm}|^2\ge|\nabla|W^{\pm}||^2
\]
at every point where $|W^{\pm}|\neq 0$, we get
\[
\frac 12\Delta_f|W^{\pm}|^2\ge|\nabla|W^{\pm}||^2+2\lambda|W^{\pm}|^2-\sqrt{6}|W^{\pm}|^3-\frac 12\langle(\overset{\circ}{\mathrm{Ric}}\circ\overset{\circ}{\mathrm{Ric}})^{\pm},W^{\pm}\rangle.
\]
Since
\[
\frac 12\Delta|W^\pm|^2=|\nabla|W^\pm||^2+|W^\pm|\Delta|W^\pm|,
\]
then we have
\[
|W^\pm|\Delta|W^\pm|\ge 2\lambda|W^{\pm}|^2-\sqrt{6}|W^{\pm}|^3-\frac 12\langle(\overset{\circ}{\mathrm{Ric}}\circ\overset{\circ}{\mathrm{Ric}})^{\pm},W^{\pm}\rangle+\frac 12\langle\nabla f, \nabla|W^{\pm}|^2\rangle.
\]
Since $M^n$ is closed, integrating the above inequality and integrating by parts over $M^n$,
we have
\begin{equation}
\begin{aligned}\label{inteineqs}
0&\ge\int_M|\nabla |W^\pm||^2dv+2\lambda\int_M|W^{\pm}|^2dv-\sqrt{6}\int_M|W^{\pm}|^3dv\\
&\quad-\frac 12\int_M\langle(\overset{\circ}{\mathrm{Ric}}\circ\overset{\circ}{\mathrm{Ric}})^{\pm},W^{\pm}\rangle dv
-\frac 12\int_M |W^{\pm}|^2 \Delta fdv.
\end{aligned}
\end{equation}
Note that Cao and Tran (Corollary 5.8 in \cite{[CaTr]}) observed
\[
\int_M\langle(\overset{\circ}{\mathrm{Ric}}\circ\overset{\circ}{\mathrm{Ric}})^{\pm},W^{\pm}\rangle dv
=4\int_M|\delta W^{\pm}|^2dv,
\]
where $\delta$ is the divergence. Hence the second assumption of theorem in fact is
\[
\int_M\langle(\overset{\circ}{\mathrm{Ric}}\circ\overset{\circ}{\mathrm{Ric}})^{\pm},W^{\pm}\rangle dv
\le\frac{1}{2}\int_M{\R}|W^{\pm}|^2dv.
\]
Using this, \eqref{inteineqs} becomes
\begin{align*}
0&\ge\int_M|\nabla |W^\pm||^2dv+2\lambda\int_M|W^{\pm}|^2dv-\sqrt{6}\int_M|W^{\pm}|^3dv\\
&\quad-\frac{1}{4}\int_M{\R}|W^{\pm}|^2dv-\frac 12\int_M |W^{\pm}|^2 \Delta fdv.
\end{align*}
Using the equality of shrinkers
\[
\Delta f=4\lambda-\R,
\]
we further have
\begin{equation}\label{evoineqty}
0\ge\int_M|\nabla |W^\pm||^2dv-\sqrt{6}\int_M|W^{\pm}|^3dv+\frac{1}{4}\int_M{\R}|W^{\pm}|^2dv.
\end{equation}
By the Sobolev inequality of shrinker \eqref{sobo} by letting $\varphi=|W^{\pm}|$,
\[
\int_M|\nabla |W^{\pm}||^2dv\ge\frac{e^{\frac{\mu}{2}}}{4C(4)}\left(\int_M|W^{\pm}|^4dv\right)^{\frac{1}{2}}
-\frac 14\int_M\R |W^{\pm}|^2dv,
\]
then \eqref{evoineqty} can be simplified as
\begin{align*}
0&\ge\frac{e^{\frac{\mu}{2}}}{4C(4)}\left(\int_M|W^{\pm}|^4dv\right)^{\frac{1}{2}}-\sqrt{6}\int_M|W^{\pm}|^3dv.
\end{align*}
Using the H\"{o}lder inequality,
\begin{align*}
0&\ge\left[\frac{e^{\frac{\mu}{2}}}{4C(4)}-\sqrt{6}\left(\int_M|W^{\pm}|^2dv\right)^{\frac{1}{2}}\right]
\left(\int_M|W^{\pm}|^4dv\right)^{\frac{1}{2}}.
\end{align*}
By the first assumption of theorem, we immediately get $W^{\pm}\equiv0$. Finally we apply
the classification of Chen and Wang \cite{[ChWa]} to conclude that the shrinker
is isometric to a finite quotient of the round sphere or the complex projective
space.
\end{proof}

In the end of this section,  following the argument of Catino \cite{[Cati]}, we
can apply the Yamabe constant to give another gap result, i.e., Theorem \ref{hagap2}
in introduction.

\begin{proof}[Proof of Theorem \ref{hagap2}]
Similar to the argument of Theorem \ref{hagap}, using the second assumption of
theorem, \eqref{inteineqs} can also be written as
\begin{align*}
0&\ge\int_M|\nabla |W^\pm||^2dv+2\lambda\int_M|W^{\pm}|^2dv-\sqrt{6}\int_M|W^{\pm}|^3dv\\
&\quad-\frac 13\int_M{\R}|W^{\pm}|^2dv-\frac 12\int_M |W^{\pm}|^2 \Delta fdv.
\end{align*}
Using the shrinker's equality
\[
\Delta f=4\lambda-\R,
\]
we obtain
\begin{equation}\label{evoineqtyd}
0\ge\int_M|\nabla |W^\pm||^2dv-\sqrt{6}\int_M|W^{\pm}|^3dv+\frac 16\int_M{\R}|W^{\pm}|^2dv.
\end{equation}
We will apply the Yamabe constant to estimate the first gradient term in the above inequality.
Recall that the Yamabe constant $Y(M,[g])$ is defined by
\[
Y(M,[g]):=\inf_{\varphi\in W^{1,2}(M)}\frac{\frac{4(n-1)}{n-2}\int_M|\nabla\varphi|^2dv_g+\int_M\R \varphi^2dv_g}
{(\int_M|\varphi|^{2n/(n-2)}dv_g)^{(n-2)/n}},
\]
where $[g]$ denotes the conformal class of $g$. As we all known, $Y(M,[g])$ is positive
on a compact manifold if and only if there exits a conformal metric in $[g]$ whose scalar
curvature is positive everywhere. Hence the compact shrinker has positive Yamabe constant
$Y(M,[g])$. If we let $\varphi=|W^{\pm}|$ on a four-dimensional compact shrinker, then
the Yamabe constant $Y(M,[g])$ implies the following inequality
\[
\int_M|\nabla |W^{\pm}||^2dv\ge\frac{Y(M,[g])}{6}\left(\int_M|W^{\pm}|^4dv\right)^{\frac{1}{2}}
-\frac 16\int_M\R |W^{\pm}|^2dv.
\]
Using this, \eqref{evoineqtyd} can be reduced to
\[
0\ge\frac{Y(M,[g])}{6}\left(\int_M|W^{\pm}|^4dv\right)^{\frac{1}{2}}
-\sqrt{6}\int_M|W^{\pm}|^3dv.
\]
By the H\"{o}lder inequality, we have
\begin{equation}\label{evqya}
0\ge\left[\frac{Y(M,[g])}{6}-\sqrt{6}\left(\int_M|W^{\pm}|^2dv\right)^{\frac{1}{2}}\right]
\left(\int_M|W^{\pm}|^4dv\right)^{\frac{1}{2}}.
\end{equation}
Recall that Gursky \cite{[Gu94]} proved the following estimate on a compact four-dimensional
 manifold
\[
\int_M\R^2dv-12\int_M|\rd|^2dv\le Y^2(M,[g]).
\]
Here this inequality is strict unless the manifold is conformally Einstein. Combining this
with the first assumption of theorem, we have
\[
6\sqrt{6}\left(\int_M|W^{\pm}|^2dv\right)^{\frac{1}{2}}\le Y(M,[g]).
\]
Combining this with \eqref{evqya} we conclude that $W^{\pm}\equiv 0$ or $(M,g)$ is conformally Einstein.

When $W^{\pm}\equiv0$, by Theorem \ref{hagap}, $(M^4,g, f)$ is isometric to a finite quotient of
the round sphere or the complex projective space.

When $(M,g)$ is conformally Einstein, it is naturally
Bach flat (see Proposition 4.78 in \cite{[Be]}) and hence is Einstein (see Theorem 1.1 in \cite{[CaCh]}).
Now since $(M^4,g)$ is Einstein, combining the first pinching condition of theorem and a gap
result of Gursky and Lebrun (see Theorem 1 in \cite{[GL]}), we also get $W^{\pm}\equiv0$ and hence
$(M^4,g, f)$ is also isometric to a finite quotient of the round sphere or the complex projective space.
\end{proof}

\bibliographystyle{amsplain}

\end{document}